\documentclass[11pt]{amsart}
\usepackage{amssymb}
\usepackage{amsmath}
\usepackage{bbm}
\usepackage{graphicx} 
\usepackage[english]{babel} 
\usepackage{float}
\usepackage{mathtools}
\usepackage{hyperref}
\usepackage[noadjust]{cite}
\usepackage[margin=3.2cm]{geometry}
\hypersetup{ 
    colorlinks=true,       % false: boxed links; true: colored links
    linkcolor=blue,          % color of internal links
    citecolor=blue,        % color of links to bibliography
    filecolor=blue,      % color of file links
    urlcolor=blue           % color of external links
}

\usepackage{pstricks} % Compile using XeLatex, View PDF
\usepackage{comment}
\usepackage{pst-all}
\usepackage{epsfig}
\usepackage{pst-grad} % For gradients
\usepackage{pst-plot} % For axes
\usepackage[space]{grffile} % For spaces in paths
\usepackage{etoolbox} % For spaces in paths
\usepackage{color} %for color text
\makeatletter % For spaces in paths
\patchcmd\Gread@eps{\@inputcheck#1 }{\@inputcheck"#1"\relax}{}{}
\makeatother

\usepackage[all]{xy}
\usepackage{mathrsfs}
\usepackage{graphics}
\usepackage{amsthm}

\theoremstyle{plain}\newtheorem{theorem}{Theorem}[section]\newtheorem{Theorem}{Theorem}\newtheorem{proposition}[theorem]{Proposition}\newtheorem{lemma}[theorem]{Lemma}

\theoremstyle{definition}\newtheorem{remark}[theorem]{Remark}%\newtheorem{construction}[subsection]{Construction}

\def\C{\mathbb{C}}\def\N{\mathbb{N}}\def\Z{\mathbb{Z}}\def\R{\mathbb{R}}\def\ZZ2{\mathbb{\Z/ 2\Z}}
\def\sb{\subset}\def\lb{\langle}\def\rb{\rangle}\def\ot{\otimes}\def\t{\times}\def\sm{\setminus}
\def\c{\gamma}\def\v{\varphi}\def\a{\alpha}\def\b{\beta}\def\d{\delta}\def\e{\epsilon}\def\s{\sigma}\def\De{\Delta}\def\La{\Lambda}\def\la{\lambda}\def\la{\lambda}\def\p{\partial}
\def\ov{\overline}\def\wt{\widetilde}
\def\ad{\text{ad}}
\def\sl2{\mathfrak{sl}_2}
\def\su2{\mathfrak{su}(2)}
\def\Aut{\text{Aut}\,}
\def\id{\text{id}}\def\Ker{\text{Ker}\,}

\def\deg{\mathrm{deg}}

\def\mf{\mathfrak}
\def\Uq{U_q(\mf{sl}_2)}\def\Uqq{\overline{U}_q(\sl2)}\def\Usl2{U_q(\sl2)}\def\usl2{\wt{U}_q(\sl2)}\def\Uq{U_q(\mf{sl}_2)}\def\uq{\mathfrak{u}_q(\sl2)}

\def\Rep{\text{Rep}}

\def\tt{\theta}
\def\kk{\mathbb{K}}\def\Svect{\text{SVect}_{\kk}}

\def\Aut{\mathrm{Aut}}

\def\ZHrho{Z^{\rho}_{\uDH}}\def\rhoc{\rho\ot \tt}

\def\CC{\mathcal{C}}

\def\uDH{\underline{D(H)}}

\def\ZZ{\mathcal{Z}}

\def\ZHrhoc{Z_{\underline{D(H')}}^{\rhoc}}

\def\CC{\mathcal{C}}

\def\ovi{I}\def\ovj{J}\def\ovn{N}

\def\PHrhoc{P_{H}^{\rhoc}}

\def\uH{\underline{H}}\def\Ha{H_{\a}} 

\def\Hc{H_{\c}}\def\Hb{H_{\b}}

\def\DHa{D(H)_{\a}}\def\DHab{D(H)_{\a\b}}\def\DHb{D(H)_{\b}}\def\DHc{D(H)_{\c}}
\def\vara{\varphi_{\a}}\def\ad{\text{ad}}\def\uDH{\underline{D(H)}}

\def\Ha{H_{\a}}

\def\uA{\underline{A}}

\def\tt{\theta}\def\FF{\mathbb{F}}

\def\Aa{D_{\a}}\def\Ab{D_{\b}}

\def\uA{\underline{A}}
\def\ZArho{Z_{\uDH}^{\rho}}
\def\PHrhoc{P_H^{\rhoc}}

\def\TT{\mathcal{T}}

\def\Auto{\mathrm{Aut}^0}

\def\ggg{\mathfrak{g}}
\def\gg{\boldsymbol{g}}

\begin{document}

\title[Genus bounds for twisted quantum invariants]{Genus bounds for twisted quantum invariants}
\author{Daniel L\'opez Neumann and Roland van der Veen}
\email{dlopezne@indiana.edu, r.i.van.der.veen@rug.nl}

\maketitle

\begin{abstract}
By twisted quantum invariants we mean polynomial invariants of knots in the three-sphere endowed with a representation of the fundamental group into the automorphism group of a Hopf algebra $H$. These are obtained by the Reshetikhin-Turaev construction extended to the $\Aut(H)$-twisted Drinfeld double of $H$, provided $H$ is finite dimensional and $\N^m$-graded.
\medskip

We show that the degree of these polynomials is bounded above by $2g(K)\cdot d(H)$ where $g(K)$ is the Seifert genus of a knot $K$ and $d(H)$ is the top degree of the Hopf algebra. When $H$ is an exterior algebra, our theorem recovers Friedl and Kim's genus bounds for twisted Alexander polynomials. When $H$ is the Borel part of restricted quantum $\sl2$ at an even root of unity, we show that our invariant is the ADO invariant, therefore giving new genus bounds for these invariants.

\begin{comment}
The Reshetikhin-Turaev construction extended to the twisted Drinfeld double of an $\N^m$-graded Hopf algebra $H$ leads to ``twisted" quantum invariants. These are $m$-variable polynomial invariants of knots in the three-sphere endowed with a representation of the fundamental group into $\Aut(H)$. We show that the degree of these polynomials is bounded above by $2g(K)\cdot d(H)$ where $g(K)$ is the Seifert genus of $K$ and $d(H)$ is the top degree of the Hopf algebra. When $H$ is an exterior algebra, the invariant is twisted Reidemeister torsion and our theorem recovers Friedl and Kim's genus bounds for twisted Alexander polynomials. When $H$ is the Borel part of restricted quantum $\sl2$ at a $2p$-th root of unity, we show that our invariant is the ADO invariant, therefore giving new genus bounds for these invariants. %When $H$ is the Borel part of quantum $\mathfrak{sl}_3$ at a 4-th root of unity, our invariant gives the correct genus for the Conway and Kinoshita-Terasaka knots.
\end{comment}
\end{abstract}

%\tableofcontents

\section[Introduction]{Introduction}

%\item Relation to Melvin-Morton-Rozansky stuff? Alex. poly comes from colored Jones.

%\item Generalize to infinite dim $H$ so we can use generic Taft. Statement of Thm 1? MAYBE: if $H$ is the infinite dim Borel and we twist by degree, the invariant is the $F_K$ invariant of Gukov-Manolescu? From Rinat's talk I remember that infinite Taft gave $ADO/alex(x^p)$. But with the present construction we could degree twist this and get something that at least looks like $F_K$ (power series in $q$, poly in $t$).

%\item Computations: polynomial time version? 
%\item Cross-cap number.
%\item Look at GUKOV's paper on Coulomb: ADO for higher rank Hopf algebras doesn't exist because rep theory is difficult (why??). But it is absolutely trivial in our setting! Moreover, for a knot, this would be a multivariable poly with as many variables as the rank. The genus theorem should be generalized to $\Z^{\oplus N}$-graded Hopf algebras.

%\item Is the holonomicity of ADO easy to see from this perspective?

%\item So far we have a Hopf algebra proof of this. A category theory proof would be nice, and the input could be any category with a non-necessarily strict action. Moreover, we should be able to see the proof in the ADO case directly from the definition of ADO.
%\item Can we find a Seifert matrix model for quantum invariants? One like the Alex poly but det is replaced by Hopf integrals.. 

\def\ss{\mathfrak{s}}

% Topological content of qtm invariants? Difference ss non-ss in this respect?

% Def. of Seifert genus

% Is it necessary to mention Turaev's G-crossed now?

\def\sl{\mathfrak{sl}}
\subsection{Background} Quantum invariants, as developed by Reshetikhin and Turaev \cite{RT1, RT2}, are topological invariants of knots and 3-manifolds built from the representation theory of quantum groups, or more generally, the theory of braided monoidal categories. The Jones polynomial is the special case when the quantum group is that associated to the Lie algebra $\mathfrak{sl}(2,\C)$. When the braided categories involved are semisimple, or more precisely modular, these invariants can be nicely packaged into what is called a 3-dimensional topological quantum field theory (TQFT), a mathematical notion that formalizes the various cut-and-paste properties of quantum invariants \cite{Turaev:BOOK2}. As the name suggests, these objects originate from physics, namely from Witten's interpretation of the Jones polynomial in terms of quantum field theory \cite{Witten:quantum}.
%Chern-Simons theory, rational conformal field theory, ...

\medskip
When the parameter $q$ of the quantum group $U_q(\mathfrak{g})$ is a root of unity, the corresponding representation category becomes non-semisimple and one can find continuous families of simple modules, such as Verma modules with complex highest weights. When $\ggg=\sl_2$ this gives rise to a polynomial invariant, called the ADO invariant, after Akutsu-Deguchi-Ohtsuki \cite{ADO}. These ``non-semisimple quantum invariants" have been much less studied than their semisimple counterparts though, partially because extending these to 3-manifolds and eventually TQFTs is much more complicated and for more than 20 years they lacked a clear physical interpretation. Some of these hurdles have been overcome in the work of Geer-Patureau-Mirand and collaborators \cite{GPT:modified, CGP:non-semisimple}, leading to TQFTs for ADO invariants with interesting topological features \cite{BCGP}. Moreover, connections between non-semisimple invariants and the physics of vertex operator algebras and logarithmic conformal field theory have recently been found \cite{Gukov:Coulomb,CDGG:QFT}.
\medskip

However, even though quantum invariants have nice cut-and-paste (TQFT) properties and physically interesting interpretations, their topological content remains mysterious. For instance, the Jones polynomial (as well as HOMFLY and Kauffman polynomials) is not clearly related to the Seifert genus of a knot, fibredness or sliceness. This is in sharp contrast to twisted Alexander polynomials (or equivalently, twisted Reidemeister torsion), which are related to all these topological properties \cite{FV:survey} and even to hyperbolic geometry \cite{Porti:survey}. Relations between colored Jones polynomials and hyperbolic geometry or $SL(2,\C)$-character varieties are still major open conjectures in the field. %At least one of these conjectures, namely the AJ conjecture, can also be extended to the non-semisimple ADO invariants \cite{BDGG:ADO}. {\color{blue} Is this worth mentioning?}

\medskip

In order to find which aspects of the theory of braided monoidal categories capture interesting topology of knot complements, a first natural question is: which ``structure" behind the Alexander polynomial and Reidemeister torsion is making it capture so much topological information? The answer is very simple: these are invariants of a covering space of the knot complement $X_K=S^3\sm K$, or in other words, invariants of a pair $(K,\rho)$ where $\rho:\pi_1(X_K)\to G$ is a homomorphism into a group $G$. For instance, the Alexander polynomial comes from the covering space corresponding to the projection $\pi_1(X_K)\to H_1(X_K)=\Z$. Most applications depend on these invariants being polynomials, for which one needs $\Z\sb G$, in particular, $G$ is an infinite group. 
\medskip

It turns out that there is an extension of the theory of quantum invariants, due to Turaev \cite{Turaev:homotopy}, that produces invariants of such pairs $(K,\rho)$. This extension starts with a $G$-crossed braided monoidal category, essentially a $G$-graded category with a $G$-action and a compatible braiding, and produces invariants of tangles endowed with a representation $\rho:\pi_1(X_T)\to G$, where $X_T=\R^2\t[0,1]\sm T$. The $G$-action condition turns out to pose a severe restriction: a semisimple monoidal category has only finitely many tensor autoequivalences up to isomorphism \cite{ENO:fusion}. In light of this, we believe it is natural to consider non-semisimple $G$-braided monoidal categories where $G$ is some infinite group. 

\medskip
In previous work, the first author \cite{LN:TDD} considered a special class of such categories: relative Drinfeld centers of crossed products $\Rep(H)\ltimes \Aut(H)$ where $H$ is a finite dimensional Hopf algebra and $\Aut(H)$ its group of Hopf algebra automorphisms, or equivalently, $\Aut(H)$-twisted Drinfeld doubles \footnote{From now on we refer to this simply as a twisted Drinfeld double as in \cite{Virelizier:Graded-QG}. Note that the term ``twisted Drinfeld double" has another common meaning in the literature.} as defined by Virelizier \cite{Virelizier:Graded-QG}. The reason to use such categories is that the twisted quantum invariants they define can be seen as ``Fox calculus deformations" of the invariants coming from the usual Drinfeld double $D(H)$ so, in some sense, they do retain some covering space theory. Moreover, it was shown in \cite{LN:TDD, LN:twisted} that the $SL(n,\C)$-twisted Reidemeister torsion of knot complements is obtained from the twisted Drinfeld double of an exterior algebra. Thus, one may expect that these quantum invariants contain topological information generalizing that of torsion.
\medskip
%When $H$ is $\Z$-graded, we have $\C^{\t}\sb\Aut(H)$ and we can use this action to define knot polynomials.
%Advantage of using Hopf and not categories: non-abelian case is clearer!
%Could also say: the twisted construction allows ``Habiro's program" to actually work in a concise way (not with the weird filtrations of Habiro).

In this work, we show that the twisted quantum invariants of \cite{LN:TDD} do indeed contain easily readable topological information about knot complements, namely, lower bounds to the Seifert genus. Moreover, we show that these invariants contain the ADO invariants, therefore giving genus bounds for these.
\medskip

% Could mention invariants Z_D(H) and then say that G-crossed extensions have the effect of deforming these invariants.
\subsection{Main theorem}

To state our main theorem, we briefly recall the construction of \cite{LN:TDD}. Let $H$ be a finite dimensional Hopf algebra over a field $\kk$ such that the Drinfeld double $D(H)$ is ribbon. The twisted Drinfeld double $\uDH$ of $H$ is a family, indexed by $\a\in\Aut(H)$, of deformations $D(H)_{\a}$ of the usual Drinfeld double of $H$ \cite{Virelizier:Graded-QG}.
If $T$ is a framed, oriented, $n$-component open tangle, the twisted Drinfeld double leads to a deformation $$\ZHrho(T)\in (H^*\ot H)^{\ot n}$$ of the usual universal invariant of tangles $Z_{D(H)}(T)$ (see e.g. \cite{Habiro:bottom-tangles}) that depends on a homomorphism $\rho:\pi_1(X_T)\to \Aut(H)$. %\footnote{Actually, to define $\ZHrho(T)$, we need that $\rho$ takes values in a subgroup $G\sb\Aut(H)$ where the homomorphism $r_H:\Aut(H)\to\kk^{\t}$ characterized by $\a(\La_l)=r_H(\a)\La_l$ for each $\a\in\Aut(H)$ has a square root.}. 
This invariant is essentially a special case of Turaev's construction \cite{Turaev:homotopy}. Now, if in addition $H$ is $\N^m$-graded, there is a canonical automorphism $\tt\in\Aut(H')$, where $H'=H\ot_{\kk}\FF$ and $\FF=\kk[t_1^{\pm 1/2},\dots,t_m^{\pm 1/2}]$, defined by $\tt(x)=t_1^{k_1}\dots t_m^{k_m}x$ for $x\in H$ homogeneous of degree $(k_1,\dots,k_m)$. Then any $\rho$ as above extends to $\rhoc:\pi_1(X_T)\to\Aut(H')$. This ``degree twisted" extension generalizes a similar procedure in twisted Reidemeister torsion theory (e.g. \cite{FV:survey}). We thus get an invariant $$\ZHrhoc(T)\in (H'^*\ot_{\FF} H')^{\ot n}$$ defined for any $\rho:\pi_1(X_T)\to\Auto_{\N^m}(H)$ (automorphisms preserving the degree and fixing the left cointegral of $H$). If $K_o$ is a framed, oriented, $(1,1)$-tangle whose closure is a knot $K$, then we define $$P_H^{\rhoc}(K)=\la\cdot\e_{D(H')}(\ZHrhoc(K_o))\in\kk[t_1^{\pm 1},\dots,t_m^{\pm 1}]$$ where $\la\in\FF^{\t}$ is a normalization factor making this an invariant of the unframed $(K,\rho)$. If $\rho\equiv 1$, $P_H^{\tt}(K)$ is a polynomial invariant of the oriented knot $K$ with no additional structure, but one still has to think that there is a canonical abelian representation $H_1(X_K)\to \Aut(H')$ sending the oriented meridian to $\tt$.

\medskip

If $\PHrhoc(K)=\sum_{\ov{k}} a_{\ov{k}}t_1^{k_1}\dots t_m^{k_m}$ where $a_{\ov{k}}\in \kk$ and $\ov{k}=(k_1,\dots,k_m)\in\Z^m$, we define $\deg \ \PHrhoc(K)$ as $\max\{k_1+\dots+k_m \ | \ a_{\ov{k}}\neq 0\}-\min\{k_1+\dots+k_m \ | \ a_{\ov{k}}\neq 0\}$.
%If $H$ is $\N^m$-graded, then it is also $\N$-graded by $H_n=\bigoplus_{i_1+\dots+i_m=n}H_{(i_1,\dots,i_m)}$. 
We denote $d(H)=\max\{i_1+\dots+i_m \ | \ H_{(i_1,\dots,i_m)}\neq 0\}$ where $H_{(i_1,\dots,i_m)}$ is the component of degree $(i_1,\dots,i_m)\in\N^m$ of $H$.

\begin{Theorem} {\bf(Genus bound)}
\label{Thm: main thm}
Let $H$ be a $\N^m$-graded Hopf algebra of finite dimension. Let $K\sb S^3$ be a knot and $\rho:\pi_1(S^3\sm K)\to\Auto_{\N^m}(H)$ a homomorphism. Then
\begin{align*}
\deg \ \PHrhoc(K)\leq 2 g(K)\cdot d(H)
\end{align*}
where $g(K)$ is the Seifert genus of $K$.
\end{Theorem}

In fact the proof shows that a similar bound holds for $\ZHrhoc(K_o)$.
As mentioned previously, the motivation for this theorem comes from the Fox calculus formulas of \cite{LN:twisted, LN:TDD}. The actual proof given here makes no mention to Fox calculus, it is instead based on an argument of Habiro \cite{Habiro:bottom-tangles} but generalized to include the representation $\rhoc$, which is essential for our theorem. We illustrate the theorem directly for (even) twist knots, which have genus one, in Subsection \ref{subs: twist knots}. Note that this theorem is in contrast with the Jones polynomial whose degree coincides with the crossing number for twist knots (since they are alternating), hence the degrees are not bounded. 
\medskip

We now list various special cases of our theorem.

\subsection{Twisted Reidemeister torsion} When $H=\La(\C^n)$ is an exterior algebra (with $\C^n$ concentrated in degree one), one has $\Auto_{\N}(H)\cong SL(n,\C)$, $d(H)=n$ and $P_{\La(\C^n)}^{\rhoc}(K)$ coincides with the $SL(n,\C)$-twisted relative Reidemeister torsion $\tau^{\rhoc}(X_K,\mu)$ where $\mu\sb \p X_K$ is a meridian \cite{LN:TDD}. This torsion is essentially equivalent to the twisted Alexander polynomial $\De_K^{\rho}(t)$ of Lin and Wada \cite{Lin:representations, Wada:twisted}. The above theorem implies $$\deg \ \tau^{\rhoc}(X_K,\mu)\leq 2g(K)n.$$
Since $\tau^{\rhoc}(X_K,\mu)=\det(t\rho(\mu)-I_n)\tau^{\rhoc}(X_K)$ this is equivalent to $\deg \ \tau^{\rhoc}(X_K)\leq n(2g(K)-1)$ which is a special case of a result of Friedl and Kim \cite{FK:Thurston}.
\medskip

\def\qq{\zeta_p}

\subsection{ADO invariants} Now let $H_p$ be the Borel part of the restricted quantum group $\overline{U}_{\qq}(\sl_2)$ at a $2p$-th root of unity $\qq$. Then $H_p$ has finite dimension, it is $\N$-graded (with $d(H_p)=p-1$) and $D(H_p)$ is ribbon, so we have a polynomial invariant $P_{H_p}^{\tt}(K)\in \Z[\qq^2][t^{\pm 1}].$ On the other hand, the unrolled restricted quantum group $\overline{U}^H_{\qq}(\sl_2)$ leads to the ADO invariant $ADO_p(K,t)\in\Z[\qq^2][t^{\pm 1}]$ \cite{ADO}. We use the normalization of the ADO invariant in which $ADO_2(K,t)=\De_K(t)$.

\begin{Theorem}
\label{Thm: ADO}
We have $$P_{H_p}^{\tt}(K)=ADO_p(K,t)$$
hence $$\deg \ ADO_p(K,t)\leq 2g(K)(p-1).$$
\end{Theorem}

As an example, the $ADO_4$ invariants of the first $6$ twist knots (denoted by their usual names in knot tables) are given as follows:
\medskip

\begin{center}
\begin{tabular}{ c|c }
 \hline
 $K$ & $ADO_{\sl_2,4}(K,t)$  \\ 
 \hline
 $3_1$ &  $(1 - 2 i) + t^{-3} + it^{-2} - (1 + i)t^{-1} + (1 + i) t + i t^2 - t^3$  \\
 $4_1$ &$7 + it^{-3} - 3t^{-2} - 6 it^{-1} + 6 i t - 3 t^2 - i t^3$  \\
 $5_2$ & $(-5 + 10 i) - (2 + 2 i)t^{-3} + (3 - 5 i)t^{-2} + (
 8 + 4 i)t^{-1} - (8 + 4 i) t + (3 - 5 i) t^2 + (2 + 2 i) t^3$     \\
$6_1$&  $(7 + 10 i) - (2 - 2 i)t^{-3} - (3 + 5 i)t^{-2} + (
 8 - 4 i)t^{-1} - (8 - 4 i) t - (3 + 5 i) t^2 + (2 - 2 i) t^3$    \\
 $7_2$ & $(-5 + 4 i) - (2 + i)t^{-3} + (3 - 2 i)t^{-2} + (
 2 + 4 i)t^{-1} - (2 + 4 i) t + (3 - 2 i) t^2 + (2 + i) t^3$  \\ 
 $8_1$ & $(13 + 2 i) - (1 - 2 i)t^{-3} - (6 + i)t^{-2} + (
 1 - 11 i)t^{-1} - (1 - 11 i) t - (6 + i) t^2 + (1 - 2 i) t^3$  \\ 
 \hline
\end{tabular}
\end{center}
\medskip

These all have degree $\leq 6$, since $g(K)=1$ for these knots, this matches with our theorem. Formulas for ADO invariants of double twist knots (genus 1) and torus knots $T_{(2,2k+1)}$ (genus $k$) are given in \cite{BH:Habiro}. All these satisfy the genus bound above.

\def\gg{\mathfrak{g}}
\medskip

\subsection{The higher rank case} More generally, for a simple Lie algebra $\ggg$ of rank $m$ and a primitive root of unity $\qq^2$ of order $p$ (say, $p$ coprime to the determinant of the Cartan matrix), our construction provides a polynomial invariant $P_{\ggg,p}(K)\in \Z[\qq][t_1^{\pm 1},\dots,t_m^{\pm 1}]$ that gives a lower bound to the genus. For instance $$\deg \ P_{\sl_{N+1},p}(K)\leq 2g(K)(p-1)\frac{1}{6}N(N+1)(N+2),$$
which generalizes the above bound for ADO to all $N$. The invariant $P_{\ggg,p}$ is obtained by letting $H$ be the Borel part of the corresponding restricted quantum group (which is $\N^m$-graded) in our construction. We expect that this invariant recovers other ``ADO-like" invariants defined in the literature, e.g. \cite{Harper:sl3-invariant}, \cite{Bishler:overview}, thus implying genus bounds for all of them. For instance, the $\sl_3$-ADO invariant at $p=2$ (so $\qq=i$) defined in \cite{Harper:sl3-invariant} is given as follows for the above twist knots: 
\medskip

\begin{center}
\begin{tabular}{ c|c }
 \hline
 $K$ & $ADO_{\sl_3,2}(K,x,y)$  \\ 
 \hline
 $3_1$ & $1 + x^{-2} - 2x^{-1} - 2 x + x^2 + y^{-2} + x^{-2} y^{-2} - 
 x^{-1} y^{-2} - 2y^{-1} - x^{-2} y^{-1} $\\
 & $+ 2x^{-1} y^{-1} + xy^{-1} - 2 y + yx^{-1} + 2 x y - 
 x^2 y + y^2 - x y^2 + x^2 y^2 $  \\
 \hline
 $4_1$ & $25 + x^{-2} - 12x^{-1} - 12 x + x^2 + y^{-2} + x^{-2} y^{-2} - 3
 x^{-1} y^{-2} - 12y^{-1} - 3x^{-2} y^{-1} $\\
& $ + 12x^{-1} y^{-1} + 3 xy^{-1} - 12 y + 3 yx^{-1} + 
 12 x y - 3 x^2 y + y^2 - 3 x y^2 + x^2 y^2$ \\
 \hline
 $5_2$ & $37 + 6x^{-2} - 26x^{-1} - 26 x + 6 x^2 + 6y^{-2} + 6x^{-2} y^{-2} - 10
 x^{-1} y^{-2} - 26y^{-1}  - 10x^{-2} y^{-1}$\\
& $ + 26x^{-1} y^{-1} + 10 xy^{-1} - 26 y + 10 yx^{-1} + 
 26 x y - 10 x^2 y + 6 y^2 - 10 x y^2 + 6 x^2 y^2  $ \\
 \hline
$6_1$& $  85 + 6x^{-2} - 46x^{-1} - 46 x + 6 x^2 + 6y^{-2} + 6x^{-2} y^{-2} - 14
 x^{-1} y^{-2} - 46y^{-1} - 14x^{-2} y^{-1}$\\
 & $ + 46x^{-1} y^{-1} + 14 xy^{-1} - 46 y + 14 yx^{-1} + 
 46 x y - 14 x^2 y + 6 y^2 - 14 x y^2 + 6 x^2 y^2 $ \\
 \hline
 $7_2$ &$ 97 + 13x^{-2} - 64x^{-1} - 64 x + 13 x^2 + 13y^{-2} + 13x^{-2} y^{-2} - 23
 x^{-1} y^{-2} - 64y^{-1} - 23x^{-2} y^{-1}$\\
& $  + 64x^{-1} y^{-1} + 23 xy^{-1} - 64 y + 23 yx^{-1} + 
 64 x y - 23 x^2 y + 13 y^2 - 23 x y^2 + 13 x^2 y^2 $ \\ 
 \hline
\end{tabular}
\end{center}
\medskip

Our genus bound implies $\deg \ P_{\sl_3,2}(K)\leq 8g(K)$, which is satisfied by the above polynomials. This will be studied in a separate paper.

%$\det(t\rho(m)-I)\frac{\De{\rho}(K)}{\De^{\rho}_0(K)}$. 

\def\TT{\Theta}

\subsection{Related results} In \cite{Ohtsuki:2loop}, Ohtsuki proved a genus bound for the 2-loop expansion of the Kontsevich integral, which is a 2-variable knot polynomial $\TT_K(x,y)$. No genus bound seems to be known for higher-loop expansions. The colored Jones polynomial is a specialization of the Kontsevich integral, so Ohtsuki's theorem implies a genus bound for the 2-loop part of this expansion as well. However, this is a limited bound and no further results seem to be known for colored Jones polynomials. Note that Ohtsuki's invariant satisfies $\TT_{\ov{K}}(x,y)=-\TT_{K}(x,y)$ so it vanishes on amphichiral knots, on the other hand, ADO invariants do not vanish on amphichiral knots. Now, the family of colored Jones polynomials turns out to be equivalent to the family of $ADO_p$ polynomials (for varying $p$) by a theorem of Willetts \cite{Willets:unification}. Thus, our theorem shows that non-semisimple quantum knot invariants contain more transparent topological information than their semisimple counterparts. 
\medskip

  %As shown in \cite{MW:unified}, the family of ADO invariants is equivalent to the family of colored Jones polynomials. %It could be possible that our genus bound is equivalent to Ohtsuki's in the case of simple Lie algebras.

\medskip
In the case of $\sl_2$, Bar-Natan and the second author studied the $n$-loop polynomial using $h$-adic Hopf algebra techniques \cite{BNV:perturbed}. This gives another proof of Ohtsuki's genus bound. This approach has the advantage of being polynomial-time computable but a genus bound was only given in the 2-loop case. Note that Seifert genus is in NP and co-NP \cite{Lackenby:Thurston}.

%{\color{blue} Is there an interesting question here? Smhtg like: if genus was poly-time, then deciding how big is ADO is also poly-time. But ADO detects colored Jones, which are P-hard.}

%Note that the HOMFLY poly is not bounded above by degree (counterexamples by Morton). 

\medskip

Finally, we note that there already exist topological invariants detecting the Seifert genus. On the one hand, a theorem of Friedl and Vidussi \cite{FV:Thurston} states that there is always some $n$ and a representation $\rho:\pi_1(X_K)\to U(n)$ for which $\deg \ \tau^{\rhoc}(X_K,\mu)=2g(K)n$. The authors do not know whether there is an algorithm to find such a $\rho$ though. On the other hand, knot Floer homology, which categorifies the Alexander polynomial, detects the Seifert genus \cite{OS:genus}. It is an interesting question whether one can actually detect the Seifert genus of a knot using Hopf algebra/representation theoretic techniques, e.g. by keeping $\rho\equiv 1$ and varying $H$ in our theorem.

%{\color{blue} Examples of knots for which it is very hard to find such $\rho$??}

\medskip

 %Some work of Kalfagianni-Lee \cite{KL:crosscap} on colored Jones and some other genus??

\subsection{Plan of the paper} We begin in Section \ref{sect: Hopf algs} with some Hopf algebra preliminaries. Here we recall some definitions from Hopf $G$-coalgebras, we define twisted Drinfeld doubles and study ribbon elements in these. In Section \ref{sect: quantum G-tangles} we define the invariant $\ZHrho(T)$, the polynomial invariant $\PHrhoc(K)$, and we prove some properties of these invariants. We illustrate our main theorem with the case of twist knots in Subsection \ref{subs: twist knots}. In Section \ref{sect: PROOF} we prove our genus bound (Theorem \ref{Thm: main thm}). Finally, in Section \ref{sect: ADO}, we study the case where the Hopf algebra is the $H_p$ above and prove Theorem \ref{Thm: ADO}.

\def\gg{\boldsymbol{g}}\def\zz{\boldsymbol{\zeta}}\def\rH{r_H}

\section{Hopf algebra preliminaries}\label{sect: Hopf algs}

%{\color{blue} We could avoid super-algebras by showing that twisted R-torsion comes from the bosonization of an exterior algebra.}

\def\Svect{\text{Svect}}

%The category $\Svect$ is symmetric with $\tau:V\ot W\to W\ot V$ defined by $$\tau(v\ot w)=(-1)^{|v||w|}w\ot v$$ for homogeneous $v\in V,w\in W$.\medskip

%In this paper, by a Hopf algebra we mean a Hopf object in the category of super vector spaces (we could work within general symmetric monoidal categories, but we won't do this). Thus, a bialgebra is an algebra $(H,m_H,1)$ in $\Svect$ together with algebra morphisms $\De:H\to H\ot H,\e:H\to\kk$ and an antipode $S:H\to H$ and satisfying the usual axioms. The only difference in the super case is that the algebra structure of $H\ot H$ involves the symmetry of $\Svect$, that is $$m_{H\ot H}=(m_H\ot m_H)\circ(\id_H\ot\tau\ot\id_H)$$, and $\De$ is an algebra morphism with respect to this algebra structure. As a consequence, the antiautomorphism property of the antipode also involves signs namely $S(xy)=(-1)^{|x||y|}S(y)S(x)$.\medskipSings in $H^*$.

For simplicity, we work over an algebraically closed field $\mathbb{K}$ of characteristic $0$. In all that follows, $G$ denotes a group. Hopf algebras will be assumed finite dimensional unless otherwise stated.

\subsection{Hopf algebras} For basic definitions on Hopf algebras, see e.g. \cite{Radford:BOOK}. We denote multiplication, coproduct, unit, counit and antipode of a Hopf algebra $H$ over $\kk$ by $m,\De,1,\e,S$ respectively. We will employ Sweedler's notation for the coproduct of a Hopf algebra, that is, $$\De(x)=x_{(1)}\ot x_{(2)}, (\De\ot\id)\De(x)=x_{(1)}\ot x_{(3)}\ot x_{(3)},$$ etc. The dual $H^*$ is also a Hopf algebra with $$\lb p\cdot q, x\rb=\lb p ,x_{(1)}\rb \lb q, x_{(2)}\rb$$ for each $p,q\in H^*$ and $x\in H$. Here $\lb \ , \ \rb$ denotes the usual pairing and $\lb p,x\rb =\lb x, p\rb$ for any $p\in H^*, x\in H$. We denote by $\Aut(H)$ the group of Hopf algebra automorphisms of $H$ and $\Auto(H)$ the subgroup of $\Aut(H)$ fixing a non-zero cointegral (either left or right, see Subsection \ref{subs: ribbon elements} below for definitions).
\medskip

A Hopf algebra $H$ is $\N^m$-graded if $H=\bigoplus_{\ovi\in\N^m} H_{\ovi}$ with $m(H_{\ovi}\ot H_{\ovj})\sb H_{\ovi+\ovj}$, $\De(H_{\ovn})\sb \sum_{\ovi+\ovj=\ovn}H_{\ovi}\ot H_{\ovj}$ and $S(H_I)\sb H_I$ for each $I,J,N\in\N^m$. We denote by $\Auto_{\N^m}(H)$ the subgroup of automorphisms of $\Auto(H)$ preserving the $\N^m$-degree.

\subsection{Hopf $G$-coalgebras} 

 A {\em Hopf $G$-coalgebra} is a family $\uH=\{\Ha\}_{\a\in G}$ where each $\Ha$ is an algebra with unit $1_{\a}$ together with a family of algebra morphisms $\De_{\a_1,\a_2}:H_{\a_1\a_2}\to H_{\a_1}\ot H_{\a_2}$ for each $\a_1,\a_2\in G$, an algebra morphism $\e:H_1\to\kk$ and algebra antiautomorphisms $S_{\a}:\Ha\to H_{\a^{-1}}$ for each $\a$ satisfying graded versions of the Hopf algebra axioms, see \cite{Turaev:homotopy, Virelizier:Hopfgroup} for more details. Note that $H=H_1$ is a Hopf algebra in the usual sense (with the coproduct $\De_{1,1}$ and antipode $S_1$).
\medskip

%This definition extends to super-vector spaces in the obvious way.

A Hopf $G$-coalgebra is said to be {\em crossed} if it comes with a family $\v=\{\v_{\a,\b}\}_{\a,\b\in G}$ of algebra isomorphisms $\v_{\a,\b}:H_{\a}\to H_{\b\a\b^{-1}}$, simply denoted $\v_{\b}$, preserving the Hopf structure and satisfying $\v_{\b_2}\v_{\b_1}=\v_{\b_2\b_1}$.
\medskip

A crossed Hopf $G$-coalgebra $(\uH,\v)$ is {\em quasi-triangular} if it comes with a family $R=\{R_{\a,\b}\}_{\a,\b\in G}$ of invertible elements in $\Ha\ot\Hb$, satisfying 

\begin{enumerate}
\item $R_{\a,\b}$ is $\v$-invariant, that is, $(\v_{\c}\ot\v_{\c})(R_{\a,\b})=R_{\c\a\c^{-1},\c\b\c^{-1}}$,
    \item $R_{\a,\b}\cdot\De_{\a,\b}(x)=(\tau[(\v_{\a^{-1}}\ot \id_{\Ha})\De_{\a\b\a^{-1},\a}(x)])\cdot R_{\a,\b}$ where $\tau$ denotes permutation of two factors (with a sign in the super-case),
    \item $(\id_{\Ha}\ot\De_{\b,\c})(R_{\a,\b\c})=(R_{\a,\c})_{1\b 3}\cdot (R_{\a,\b})_{12\c}$,
\item $(\De_{\a,\b}\ot\id_{\Hc})(R_{\a\b,\c})=((\id_{\Ha}\ot\v_{\b^{-1}})(R_{\a,\b\c\b^{-1}}))_{1\b 3}\cdot (R_{\b,\c})_{\a 23}$
\end{enumerate}
 for each $\a,\b,\c\in G$, where we used the notation $X_{1 2\c}=x\ot y\ot 1_{\c}$ for any $X=\sum x\ot y\in \Ha\ot\Hb$ and similarly for $Y_{\a 23}, Z_{1\b 3}$ ($Y\in\Hb\ot\Hc, Z\in \Ha\ot\Hc$). If $(\uH,\v,R)$ is quasi-triangular, the graded Drinfeld element is $u=\{u_{\a}\}_{\a\in G}$ defined by $$u_{\a}=m_{\a}(S_{\a^{-1}}\v_{\a}\ot\id_{\Ha})\tau_{\a,\a^{-1}}(R_{\a,\a^{-1}})\in\Ha.$$

A quasi-triangular Hopf $G$-coalgebra $(\uH,\v,R)$ is {\em $G$-ribbon} if it comes with a family $v=\{v_{\a}\}_{\a\in G}$ of invertible elements $v_{\a}\in H_{\a}$ satisfying, for each $\a,\b\in G$:
\begin{enumerate}
    \item $\De_{\a,\b}(v_{\a\b})=(v_{\a}\ot v_{\b})(\tau_{\b,\a}[\v_{\a^{-1}}\ot\id(R_{\a\b\a^{-1},\a})])R_{\a,\b}$,
    \item $S_{\a}(v_{\a})=v_{\a^{-1}}$,
    \item $\v_{\a}(x)=v_{\a}^{-1}xv_{\a}$ for all $x\in \Ha$,
    \item $\v_{\b}(v_{\a})=v_{\b\a\b^{-1}}$.
\end{enumerate}

If $\gg_{\a}=v_{\a}u_{\a}$, we call $\gg=\{\gg_{\a}\}_{\a\in G}$ the $G$-pivot of $\uH$, this satisfies $S_{\a^{-1}}S_{\a}(x)=\gg_{\a}x\gg_{\a}^{-1}$ for each $x\in \Ha,\a\in G$ and is group-like in the sense that $\De_{\a,\b}(\gg_{\a\b})=\gg_{\a}\ot \gg_{\b}$.
\medskip

%The above axioms for a ribbon Hopf $G$-coalgebra imply that the category $\CC=\coprod_{\a\in G}\Rep(\Ha)$ is a {\em strict} ribbon $G$-crossed category. Here strict means that the action $\a\mapsto (F_{\a}:\CC\to\CC)$ of $G$ on $\CC$ is strict in that the natural isomorphisms $F_{\a}(X\ot Y)\cong F_{\a}(X)\ot F_{\a}(Y)$ and $F_{\a}F_{\b}\cong F_{\a\b}$ are identity maps.

\subsection{Twisted Drinfeld doubles}\label{subs: TDDs} We now define the twisted Drinfeld double of a Hopf algebra $H$, which is a quasi-triangular Hopf $\Aut(H)$-coalgebra extending the Drinfeld double construction. This was introduced by Virelizier in \cite{Virelizier:Graded-QG}, here we follow the conventions of \cite{LN:TDD} (see remark \ref{remark: rel Drinfeld center} below). 
\medskip

For each $\a\in\Aut(H)$ we define an algebra $\DHa$ as follows: $\DHa=H^*\ot H$ as a vector space and we define the product by $$(p\ot a)\cdot_{\a} (q\ot b)=\lb a_{(1)}, q_{(3)}\rb \lb S^{-1}\a^{-1}(a_{(3)}),q_{(1)}\rb p\cdot q_{(2)}\ot a_{(2)}\cdot b$$
where $p,q\in H^*, a,b\in H$. 
This is an associative algebra with unit $1_{\a}=\e_H\ot 1_H$. For each $\a,\b\in\Aut(H)$ we define a coproduct $\De_{\a,\b}:\DHab\to \DHa\ot\DHb$ by $$\De_{\a,\b}(p\ot a)=(p_{(2)}\ot a_{(1)})\ot (p_{(1)}\ot \a^{-1}(a_{(2)})).$$

\noindent For each $\a\in\Aut(H)$ we define an antipode $S_{\a}:D(H)_{\a}\to D(H)_{\a^{-1}}$ by $$S_{\a}(p\ot a)=(\e\ot\a^{-1}(S(a)))\cdot_{\a^{-1}}(p\circ S^{-1}\ot 1).$$
\noindent For each $\a,\b$ define an algebra isomorphism $\vara:D(H)_{\b}\to D(H)_{\a\b\a^{-1}}$ by $$\vara(p\ot a)=p\circ\a^{-1}\ot \a(a).$$
Finally, for each $\a,\b$ let $$R_{\a,\b}=\sum (\e\ot\a(h_i))\ot (h^i\ot 1)\in D(H)_{\a}\ot D(H)_{\b}$$
where $(h_i)$ is any basis of $H$ and $(h^i)$ is the dual basis. With all these structure maps, $\{D(H)_{\a}\}_{\a\in\Aut(H)}$ has the structure of a quasi-triangular crossed Hopf $\Aut(H)$-coalgebra. We denote $\uDH=\{\DHa\}_{\a\in\Aut(H)}$ and call it the {\em twisted Drinfeld double} of $H$. If $\phi:G\to\Aut(H)$ is a group homomorphism, then $\uDH|_G=\{D(H)_{\phi({\a})}\}_{\a\in G}$ is a quasi-triangular Hopf $G$-coalgebra in an obvious way, we call it a $G$-twisted Drinfeld double.
 \medskip
 
 \begin{remark}
 \label{remark: rel Drinfeld center}
 The category of modules of the twisted Drinfeld double is braided $\Aut(H)$-crossed equivalent to the relative Drinfeld center $\mathcal{Z}_{rel}(\Rep(H)\rtimes \Aut(H))$ of \cite{GNN:centers}, where $\Aut(H)$ acts on $\Rep(H)$ by $V\to T_{\a}(V)$ for any $V\in\Rep(H)$, where $T_{\a}(V)$ is $V$ as a vector space with $H$-action $h\cdot v:=\a^{-1}(h)v$, see \cite{LN:TDD}. 
 \end{remark}

\subsection{Ribbon elements in the double}\label{subs: ribbon elements} We will now establish under which conditions the twisted Drinfeld double has a graded ribbon element. For this, we briefly recall what happens for the usual Drinfeld double, for more details see \cite{Radford:BOOK}.
\medskip

\def\bold{\boldsymbol}
\def\bbb{\bold{\b}}
\def\bbeta{\bold{\b}}
\def\bb{\bold{b}}

Recall first that a {\em left cointegral} of a Hopf algebra $H$ is an element $\La_l\in H$ such that $x\La_l=\e(x)\La_l$ for each $x\in H$. Dually, a {\em right integral} is $\la_r\in H^*$ such that $\la_r(x_{(1)})x_{(2)}=\la_r(x)1_H$ for each $x\in H$. As a consequence of uniqueness of integrals/cointegrals, there are unique group-likes $\bold{\a}\in G(H^*),\bold{a}\in G(H)$ satisfying $$x_{(1)}\la_r(x_{(2)})=\la_r(x)\bold{a}, \hspace{1cm} \La_lx=\bold{\a}(x)\La_l$$
for all $x\in H$. These are the {\em distinguished group-likes} of $H$.

\medskip

Now, a theorem of Kauffman and Radford states that the Drinfeld double $D(H)$ is ribbon if and only if there are group-likes $\bb\in G(H),\bbb\in G(H^*)$ such that $\bb^2=\bold{a}, \bbb^2=\bold{\a}$ and $$S^2=\ad_{\bbb^{-1}}\circ\ad_{\bb}.$$
In such a case, the ribbon element is $v=u^{-1}(\bbeta\ot\bb)$ where $u$ is the Drinfeld element of $D(H)$.

\medskip

We can now give sufficient conditions for $\uDH$ to be ribbon (or rather, a restriction to a certain subgroup $G\sb \Aut(H)$). Let's suppose that $H$ satisfies the above condition making $D(H)$ ribbon. Let $G\sb\Aut(H)$ be a subgroup fixing such $\bb,\bbb$. Let $\rH:\Aut(H)\to\kk^{\t}$ be the homomorphism characterized by $\a(\La_l)=\rH(\a)\La_l$ (this is defined by uniqueness of cointegrals). Then, as shown in \cite[Prop. 2.4]{LN:TDD}, $\uDH|_G$ is $G$-ribbon if and only if the homomorphism $\rH|_G:G\to\kk^{\t}$ has a square root. In such a case, the $G$-ribbon element is given by $$v_{\a}=\sqrt{\rH(\a)}^{-1}(\id_{H^*}\ot\a)(v)$$
for any $\a\in G$, where $v$ is the ribbon element of $D(H)$. The $G$-pivot is given by $$\gg_{\a}=\sqrt{\rH(\a)}^{-1}(\bbeta\ot\bb).$$ 
The square root condition is immediately satisfied if $G=\Ker(r_H)=\Auto(H)$.

\def\DHa{D(H)_{\a}}\def\DHb{D(H)_{\b}}\def\DHab{D(H)_{\a\b}}\def\DHc{D(H)_{\c}}\def\DHd{D(H)_{\d}}
\def\DH{D(H)}

\subsection{Trivializing the crossing}\label{subs: ss-product Aa's} Suppose that $G$ is an abelian group with an homomorphism $\phi:G\to\Aut(H)$ and consider $\uDH|_G$ as defined above. Then $G$ acts on each $\DHa$ by $\b\mapsto \v_{\phi(\b)}|_{\DHa}$ and setting $A_{\a}=\kk[G]\ltimes \DHa$ we get another Hopf $G$-coalgebra $\uA=\{A_{\a}\}_{\a\in G}$. The Hopf structure is extended to $\uA$ by declaring the elements of $G$ to be group-likes. In what follows we denote $\b\in G$ by $\v_{\b}$ when considered as an element of $A_{\a}$ so that $\v_{\b}x=\v_{\phi(\b)}(x)\v_{\b}$ in $A_{\a}$ for each $x\in\DHa$. It is easy to see that $$R_{\a,\b}^A=(1\ot\v_{\a})R_{\a,\b}\in A_{\a}\ot A_{\b}$$ is an $R$-matrix for $\uA$ in (almost) the usual sense, that is, it satisfies $$R^A_{\a,\b}\cdot\De^A_{\a,\b}(x)=\De^{A,op}_{\b,\a}(x)\cdot R^A_{\a,\b}.$$ In other words, the category $\CC=\coprod_{\a\in G}\Rep(A_{\a})$ is a braided category in the usual sense. %Categorically speaking, the category $\CC$ is the equivariantization of the category $\mathcal{D}=\coprod_{\a\in G}\Rep(\DHa)$. 
It is not difficult to see that the Drinfeld element of $\uA$ is determined by $u^A_{\a}=u_{\a}\v_{\a}^{-1}$ where $u=\{u_{\a}\}$ is the Drinfeld element of $\uDH$. If $v=\{v_{\a}\}$ is a ribbon element in $\uDH$, then $\uA$ is ribbon with $v^A_{\a}=v_{\a}\v_{\a}$. It follows that the pivot of $\uA$ is the same as in $\uDH$.

\medskip

\subsection{The super algebra case} Recall that a super vector space is a vector space $V$ with a mod 2 grading, that is, a decomposition $V=V_0\oplus V_1$. The category of super vector spaces is symmetric with $\tau_{V,W}(v\ot w)=(-1)^{|v||w|}w\ot v$
where $|v|,|w|$ denotes the mod 2 degree of homogeneous elements $v\in V, w\in W$. A super Hopf algebra is a Hopf algebra $(H,m,1,\De,\e,S)$ in the category of super vector spaces. This amounts to the same axioms as for a Hopf algebra, except that the coproduct $\De$ satisfies
$$\De\circ m=(m\ot m)\circ(\id\ot\tau_{H,H}\ot\id)\circ(\De\ot\De).$$
More generally, we can talk about super Hopf $G$-coalgebras, now it is $\De_{\a,\b}$ that satisfies the above property (with $\tau_{\Hb,\Ha}$ in place of $\tau_{H,H}$). In the super case, $\Aut(H)$ denotes the automorphisms preserving the mod 2 degree. The twisted Drinfeld double of a super Hopf algebra $H$ is defined as before, but with some additional signs: 
\begin{align*}
    (p\ot a)\cdot_{\a} (q\ot b)&=(-1)^{|a_{(1)}|+|q_{(1)}|+|a_{(2)}||q_{(2)}|+|a_{(1)}||q_{(2)}|+|a_{(2)}||q_{(3)}|}\\
    &\lb a_{(1)}, q_{(3)}\rb \lb S^{-1}\a^{-1}(a_{(3)}),q_{(1)}\rb p\cdot q_{(2)}\ot a_{(2)}\cdot b.
\end{align*}

In this formula, $\lb \ , \ \rb$ represents the usual vector space pairing, but note that the sign $(-1)^{|a_{(1)}|+|q_{(1)}|}$ comes from the fact that the right evaluation/coevaluation of the category of super vector spaces is not the same as that of vector spaces.
\medskip

\section{Twisted quantum invariants of $G$-tangles}
\label{sect: quantum G-tangles}

%In this section we introduce the invariant $\ZHrho(T)$ for a $G$-tangle $(T,\rho)$. In Subsection \ref{subs: Zm graded case} we explain how to lift these invariants to polynomials, provided $H$ is $\Z^m$-graded. In Subsection \ref{subs: lemmas} we study some simple properties of this invariant. In Subsection \ref{subs: twist knots} we illustrate our construction, and our main theorem, with the case of twist knots.\medskip

\def\DD{D(H)}
\def\DDp{D(H')}
\def\Aa{\DD_{\a}}
\def\Ab{\DD_{\b}}
\def\DDP{\DDp}

During this section, we let $H$ be a finite dimensional Hopf algebra. We suppose $D(H)$ is ribbon with corresponding pivotal element $\bbeta\ot\bb$ and we let $G\sb\Aut(H)$ be such that $\uDH|_G$ is $G$-ribbon. %Consider the $G$- coalgebra $\uA=\{A_{\a}\}$ defined out of $\uDH|_G$ as in Subsection \ref{subs: ss-product Aa's}. 
We let $(h_i)$ be a basis of $H$ and $(h^i)$ be the dual basis of $H^*$.

\begin{figure}[htp!]
    \centering
    \includegraphics[width=6cm]{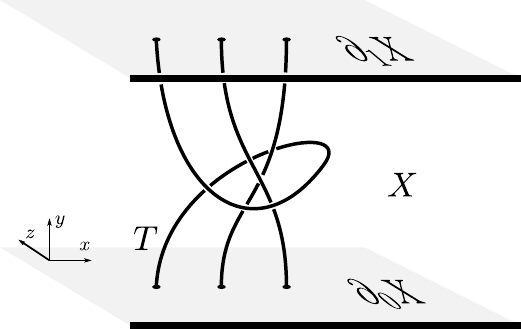}
    \caption{A three strand tangle $T$ in the block $X$.}
    \label{fig:Tangle}
\end{figure}

\subsection{$G$-tangles} Let $X=\R\t [0,1]\t [-1,\infty)$ and let $p$ be a basepoint in $\R\t[0,1]\t \{-1\}$. We think of $\R\t[0,1]\t \{0\}$ as being on the plane of the page, so the $z$-axis is transversal to the plane and the negative $z$-axis is towards the reader. We also denote $\p_i X=\R\t\{i\}\t[-1,\infty)$ for $i=0,1$. By a $G$-tangle we mean a framed, oriented tangle $T$ in $\R\t[0,1]\t(0,\infty)$ endowed with a representation $\rho:\pi_1(X_T,z)\to G$ where $X_T=X\sm T$. Two $G$-tangles $(T,\rho)$ and $(T',\rho')$ are isotopic if there is a basepoint-preserving isotopy $d_t:X\to X$, that is, $d_0=\id_X,\ d_t|_{\p X}=\id,\ d_1(T)=T'$ as framed oriented tangles, and $\rho'\circ(d_1)_*=\rho$. Here $(d_1)_*:\pi_1(X_T,p)\to \pi_1(X_{T'},p)$ is the induced map. By Van Kampen's theorem, if $(T,\rho)$ and $(T',\rho')$ are $G$-tangles satisfying that $\rho\circ j_1|_{\pi_1(\p_1X\cap X_T)}=\rho'\circ j_0|_{\pi_0(\p_0X\cap X_{T'})}$ (where $j_i$ is the homomorphism induced by the corresponding inclusion) then we can stack $(T',\rho')$ on top of $(T,\rho)$ and rescale to define a new $G$-tangle, called the composition of $(T',\rho')$ and $(T,\rho)$. Similarly, we can stack a $G$-tangle to the right of another. With these operations, $G$-tangles form (the morphisms of) a monoidal category, and in fact, a $G$-crossed ribbon category \cite{Turaev:homotopy}.
\medskip

%If $T_i$ is a component of a framed oriented $G$-tangle, then there is a well-defined class $[T_i]\in\pi_1(X_T)$ defined as follows: let $c_i$ be the core of $T_i$, that is, $T_i=c_i\t [0,1]$ as a framed tangle and let $c'_i\sb X_T$ be a close normal to $c_i$ (this uses the framing of $T_i$). Then we define $[T_i]$ as the class of a path that goes linearly from the basepoint $p$ to the beginning of $c'_i$, follows $c'_i$ up to its endpoint and then goes back to the basepoint by a linear path. 

\subsection{Twisted universal invariants of $G$-tangles}

\medskip
Let $(T,\rho)$ be a $G$-tangle with $n$ open components (and no closed components). %(see Remark \ref{remark: closed case} below for the closed case). 
We suppose the components of $T$ are ordered, say $T_1,\dots,T_n$. We define now the invariant $\ZHrho(T)\in \DD_{\a_1}\ot\dots\ot \DD_{\a_n}$
where each $\a_i=\rho(\mu_i)$ and $\mu_i$ is a meridian at the endpoint of $T_i$ as defined below. As usual, this invariant is defined from a planar diagram and then shown to be independent of it. %If $\a_1,\dots,\a_m$ represent the conjugacy classes of these components, we will now define an invariant $$\ZArho(T)\in A_{\a_1}\ot\dots\ot A_{\a_m}$$of the pair $(T,\rho)$. 
\medskip

Let $D$ be an oriented planar diagram of $T$ where at all crossings both strands are oriented upwards. We assume $D$ comes with the blackboard framing.
As usual in drawing a knot diagram we cut the projection close to each crossing to indicate where the knot passes under another strand. The resulting connected components in the plane are called the edges (arcs) of the diagram.
\medskip

We define first meridians and ``partial longitudes" associated to the edges of a diagram. For each edge $e$ of $D$, there is an element $\mu_e\in \pi_1(X_T,p)$ defined as the homotopy class of the loop that starts at the basepoint $p$, goes to the given edge via a linear path, encircles that edge once with linking number -1 and finally goes back to the basepoint by the same linear path. Then $\rho(\mu_e)\in G$, so that all edges of the diagram are labelled by elements of $G$:
\begin{figure}[H]
   \includegraphics[width=11cm]{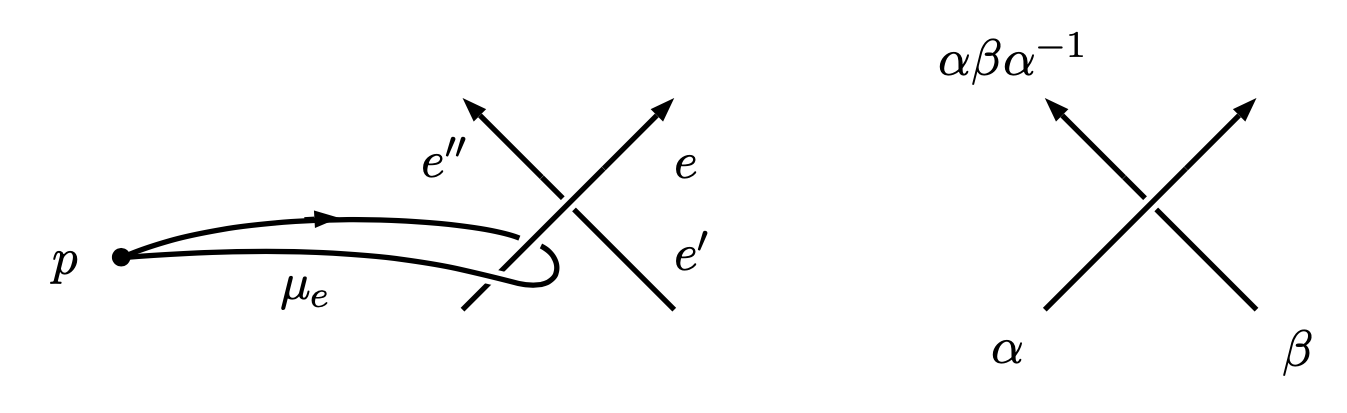}
\caption{On the right $\a=\rho(\mu_e)$ and $\b=\rho(\mu_{e'})$.}
\label{fig: crossing with labels}
\end{figure}

%Now, at each crossing of the diagram we put three ``beads", two are Hopf algebra elements and one is an isomorphism $\v_{\a^{\pm 1}}$. Depending on the type of crossing and the $G$-labels, this is done in the following way:
 \noindent For each $i=1,\dots, n$ we let $\mu_i=\mu_{e_i}$ where $e_i$ is the edge containing the endpoint of $T_i$ and we let $\a_i=\rho(\mu_i)\in G$. Now let $c_i$ be the core of $T_i$, that is, $T_i\cong c_i\t [0,1]$ as a framed tangle and let $c'_i\sb X_T$ be a close oriented normal to $c_i$ (this uses the framing of $T_i$). Then, for each edge $e\sb T_i$ of the diagram $D$ of $T$, we let $l_e\in\pi_1(X_T)$ be the ``partial longitude" defined as the homotopy class of a path that goes linearly from $p$ to the endpoint of $c'_i$, follows $c'_i$ with the opposite orientation until it reaches $e$, and finally goes back to the basepoint $p$ by a linear path:
 
 \begin{figure}[H]
     \centering
    \includegraphics[width=5cm]{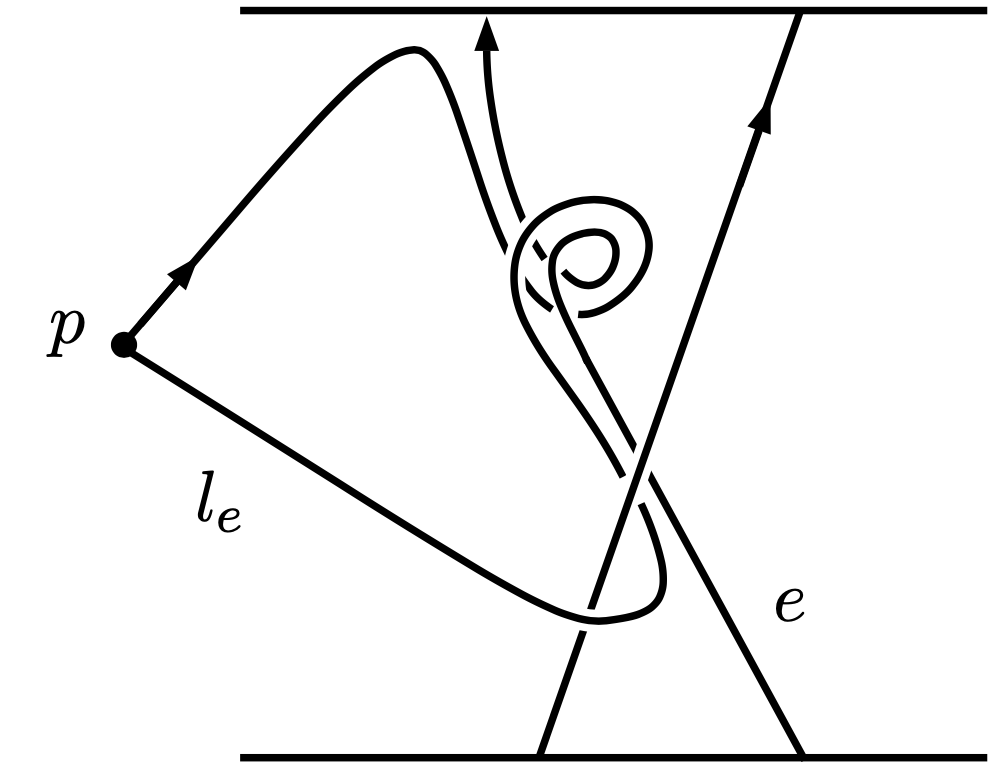}
\caption{The partial longitude $l_e$ of the edge $e$.}
\label{fig: partial longitude}
 \end{figure}

\noindent Note that for each edge $e\sb T_i$ of $D$ we have that $l_e\mu_e l_e^{-1}=\mu_{i}$. For a component $T_i$ of $T$ we denote $[T_i]=l_{e'_i}$ where $e'_i$ is the edge of $T_i$ containing the starting point of $T_i$. Note that both $(\a_1,\dots,\a_n)$ and $(\rho([T_1]),\dots,\rho([T_n]))$ are independent of the diagram $D$ of $T$.
\medskip

 We now bring in the twisted Drinfeld double. To each positive crossing of $D$ whose bottom edges are labelled by $\a,\b$ (as in Figure \ref{fig: crossing with labels} above), we associate the $R$-matrix $R_{\a,\b}\in \Aa\ot\Ab$ and the crossing isomorphism $\v_{\a}:\DD_{\b}\to \DD_{\a\b\a^{-1}}$. We represent the two factors of the $R$-matrix in the diagram by two black beads placed at the crossing, the first tensor factor on the overpass and the second on the underpass, while $\v_{\a}$ is represented by a white bead: 

\begin{figure}[H]
 \includegraphics[width=7cm]{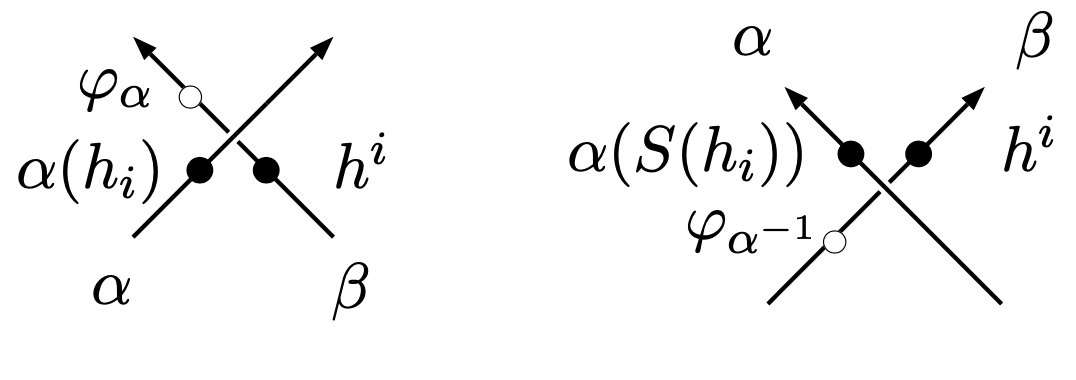}
\end{figure}

If the crossing is negative and $\a,\b$ are the labels of the edges at the top, we assign the inverse $R$-matrix $R_{\a,\b}^{-1}$ and the crossing isomorphism $\v_{\a^{-1}}$. These are again represented by two black and one white beads as above. To caps and cups labelled by $\a$ we associate the unit or the $G$-pivot as follows:

\begin{figure}[H]
  \includegraphics[width=9cm]{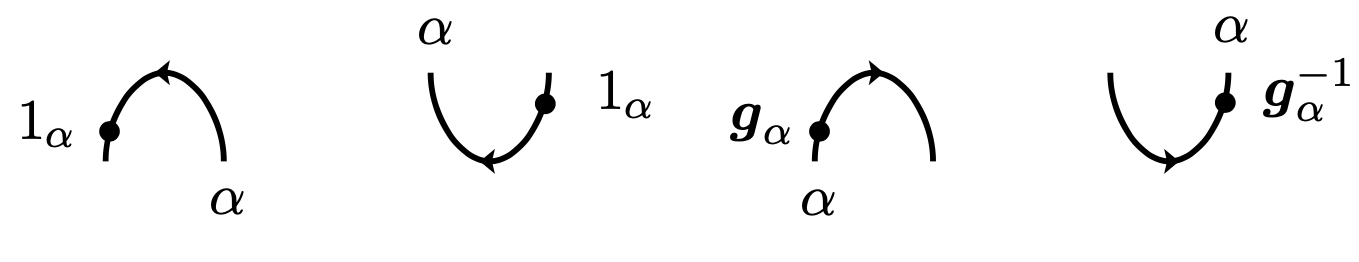}
\end{figure}

\noindent where $\gg_{\a}=\sqrt{\rH(\a)}^{-1}\bbeta\ot\bb$. With these conventions, the black beads lying over a given edge $e$ of the diagram $D$ belong to $\DD_{\rho(\mu_e)}$. Hence, we can multiply all these black beads (from right to left) as we follow the orientation along $e$. Thus, we get a single black bead $x_e\in \DD_{\rho(\mu_e)}$ for each edge of the diagram. We now follow the orientation of a component $T_i$ of $T$ and we multiply all these black beads as follows: if an edge $e''$ follows an edge $e'$ and the label of the overpass separating these is $\a$, then there is a white bead $\v_{\a^{\e}}$ in between the black beads of $e'$ and $e''$ (where $\e=\pm 1$ is the sign of the crossing). We evaluate $\v_{\a^{\e}}$ on the bead $x_{e'}\in \DD_{\rho(\mu_{e'})}$  and then we multiply this with the bead $x_{e''}\in \DD_{\rho(\mu_{e''})}$, this results in $x_{e''}\v_{\a^{\e}}(x_{e'})\in \DD_{\rho(\mu_{e''})}$. In other words, we slide the black beads following the orientation of the diagram, and whenever a black bead crosses a white bead, we evaluate the corresponding crossing isomorphism on the black bead. Note that the product of all the white beads on a given component of $T$ lying after an edge $e$ is equal to $\v_{\rho(l_e)}$. Thus, we could equally evaluate $\v_{\rho(l_e)}(x_e)$ for each edge $e$, which belongs to $\DD_{\a_i}$ for each $e\sb T_i$, and then multiply all the resulting beads on a given component (from right to left as we follow the orientation of $T_i$). 
If we do this for all components of $T$ we get an element $$\ZArho(D)\in \DD_{\a_1}\ot\dots\ot \DD_{\a_n}.$$
This turns out to be independent of the diagram $D$ chosen, so it is a topological invariant of $(T,\rho)$, denoted $\ZHrho(T)$. This is the invariant of $G$-tangles defined in \cite{LN:TDD}, we call it a {\em twisted universal quantum invariant}.
\medskip

More precisely, and to be careful with the signs in the super-case, we define the invariant $\ZArho(T)$ as follows: let $S$ be the set consisting of the overpasses and underpasses of the diagram together with the right caps and cups. Suppose $D$ has $N$ crossings and $k$ right caps or cups, so $|S|=2N+k$. Order the set of crossings of the diagram $D$ arbitrarily, say $c_1,\dots,c_N$. Order the set of right caps/cups arbitrarily too. This determines a total order of $S$ where the overpass of each crossing comes before the underpass and all caps/cups come after the crossings. We denote $S_1$ the set $S$ with this order. Let $\a_i,\b_i$ be the labels of the bottom edges (resp. top edges) at $c_i$ if the crossing is positive (resp. negative). Let $\c_i$ be the label at the $i$-th right cap/cup for $i=1,\dots,k$ and let $\e_i$ be the corresponding power of $g_{\c_i}$. For each $x\in S$, let $\DD_x$ be $\DD_{\a_i}$, $\DD_{\b_i}$ or $\DD_{\c_i}$ depending on whether $x$ is an overpass, an underpass or a cap/cup. Now, following the orientation of each component of the diagram, along with the order of the components of $T$, defines another total order on $S$, denote it $S_2$. In other words, if $d_i$ is the number of elements of $S$ corresponding to $T_i$, then the first $d_1$ elements of $S_1$ are those of $T_1$ (ordered from right to left as we follow the orientation of $T_1$), the next $d_2$ elements are those of $T_2$ ordered in the same way, and so forth. This gives a permutation isomorphism $P:\otimes_{x\in S_1}\DD_x\to \otimes_{x\in S_2}\DD_x$. If $\uDH$ is a Hopf group-coalgebra in super-vector spaces, this means that signs are being introduced. Now, if $$R_D=\left(\bigotimes_{j=1}^N R_{\a_j,\b_j}\right)\ot\left(\bigotimes_{i=1}^kg_{\c_i}^{\e_i}\right)$$ (note that this belongs to $\otimes_{x\in S_1}D_x$) then $$\ZArho(T)=\left(\bigotimes_{i=1}^n m_{\a_i}^{(d_i)}\right)\circ P\circ \left(\bigotimes_{x\in S_1}\v_{\rho(l_x)}\right)(R_D),$$
where $m_{\a_i}^{(k)}:\DD_{\a_i}^{\ot k}\to \DD_{\a_i}$ denotes iterated multiplication. Here we denote $l_x=l_{e_x}$ where $e_x$ is the edge of $D$ containing $x$ and $l_{e_x}$ is the partial longitude defined above. Note that the order chosen on the set of crossings and in the set of caps/cups is irrelevant in the super-case, since both the $R$-matrices and the pivots have degree zero.

\medskip
\begin{remark}
Our motivation to define $\ZHrho(T)$ for super Hopf algebras is that the $SL(n,\C)$-twisted Reidemeister torsion is the special case when $H$ is an exterior algebra $\La(\C^n)$ \cite{LN:TDD}, which is a super Hopf algebra.
\end{remark}

\subsection{The abelian case} Suppose $G$ is an abelian group with an homomorphism $\phi:G\to\Aut(H)$. Note that, given a diagram of a $G$-tangle $(T=T_1\cup\dots\cup T_n,\rho)$, all edges of $T_i$ will be labelled by the same $\a_i\in G$. Consider $\uDH|_G$ and the Hopf $G$-coalgebra $\uA$ defined in Subsection \ref{subs: ss-product Aa's}. Then we can define a universal invariant $Z^{\rho}_{\uA}(T)\in A_{\a_1}\ot\dots\ot A_{\a_n}$ almost in the ``usual sense": we use the $R$-matrix $R^A_{\a,\b}$ on crossings (with labels $\a,\b$ as before) and the pivot of $\uA$ (which is the same as that of $\uDH$) on caps/cups and we run the above procedure. Every bead over a component $T_i$ of a tangle $T$ belongs to the same algebra $A_{\rho([T_i])}$, so we can multiply all of these directly. Using the semidirect product relation $\a\cdot x=\v_{\a}(x)\cdot \a,\ \a\in G,\ x\in D(H)$ and that $R^A_{\a,\b}=\sum h_i\ot (h^i\cdot\v_{\a})$ we can write this invariant as 
\begin{align}
\label{eq: ZA fom ZDH}
    Z^{\rho}_{\uA}(T)=\ZHrho(T)\cdot (\v_{\rho([T_1])}\ot\dots \ot\v_{\rho([T_n])}).
\end{align}
This is because the product of all white beads over $T_i$ equals $\v_{\rho([T_i])}$.

\subsection{The $\N^m$-graded case}\label{subs: Zm graded case} In addition to the previous conditions on $H$, suppose that $H$ is $\N^m$-graded and let $G= \Auto_{\N^m}(H)$.%We now explain how to upgrade the previous invariant to a polynomial invariant in $\kk[t_1^{\pm 1/2},\dots, t_m^{\pm 1/2}]$.
\medskip

Let $\FF=\kk[t_1^{\pm 1/2},\dots,t_m^{\pm 1/2}]$ and $H'=H\ot_{\kk} \FF$. %We denote by $\tt\in\Aut(H')$ the Hopf automorphism defined by $$\tt(x)=t_1^{|x|_1}\dots t_m^{|x|_m}x$$where $(|x|_1,\dots,|x|_m)$ is the $\Z^m$-degree of $x\in H$. More generally, 
For every $\a\in\Aut_{\N^m}(H)$ and $n\in\Z$ we define an $\FF$-linear Hopf automorphism of $H'$, denoted $\a\ot\tt^n\in\Aut(H')$, by $$\a\ot\tt^n(x)=t_1^{n|x|_1}\dots t_m^{n|x|_m}\a(x)$$ where $(|x|_1,\dots,|x|_m)$ is the $\N^m$-degree of $x\in H$. It is easy to see that the map $\Aut_{\N^m}(H)\t\Z\to\Aut(H'),\ (\a,n)\mapsto \a\ot \tt^n$ is a homomorphism. In what follows we use the shorthand $\tt=\id\ot\tt^1 \in\Aut(H')$. Let $G'\sb \Aut_{\N^m}(H')$ be the subgroup of automorphisms of the form $\a\ot \tt^n$ where $\a\in G,\ n\in\Z$. Since $r_H(\a)=1$, it is easy to see that $r_{H'}(\a\ot\tt^n)=t_1^{n|\La_l|_1}\dots t_m^{n|\La_l|_m}$ (where $r_{H'}$ is as in Subsection \ref{subs: ribbon elements}) which has a square root $$\sqrt{r_{H'}(\a\ot\tt^n)}=t_1^{n|\La_l|_1/2}\dots t_m^{n|\La_l|_m/2},$$ hence $D(H')|_{G'}$ is $G'$-ribbon by \cite[Prop. 2.4]{LN:TDD}. The pivot element of $D(H')_{\a\ot\tt}$ is $$\gg_{\a\ot\tt}=t_1^{-|\La_l|_1/2}\dots t_m^{-|\La_l|_m/2}\bbeta\ot\bb.$$ %Thus, we can use this extension to define tangle invariants valued in $\FF$.
\medskip

\def\rhoc{\rho\ot\tt}
Now let $T$ be a tangle as in the previous section and let $\rho:\pi_1(X_T)\to G$ be an homomorphism. Then $\rho$ extends to a homomorphism $\rhoc:\pi_1(X_T)\to G'\sb\Aut(H')$ defined by
\begin{align*}
    \d\mapsto \rho(\d)\ot\tt^{h(\d)}
\end{align*}
for $\d\in\pi_1(X_T)$,
where $h:\pi_1(X_T)\to H_1(X_T)\to\Z$ is the homology representation. Note that this depends on the orientation of the components of $T$. When $\rho\equiv 1$ this is simply the abelian representation sending the canonical generators of $H_1(X_T)$ to $\tt$. Then the above construction leads to an invariant $$\ZHrhoc(T)\in D(H')_{\a_1\ot\tt}\ot\dots\ot D(H')_{\a_n\ot\tt}.$$ 
In what follows, we identify $D(H')_{\a\ot\tt}$ with $(H^*\ot H)[t_1^{\pm 1/2},\dots, t_m^{\pm 1/2}]$ in the obvious way.
\medskip

\def\uDHp{\underline{D(H')}}
\subsection{Knot polynomials} Suppose $H$ is $\N^m$-graded as above. If $(K_o,\rho)$ is a framed, oriented, $(1,1)$-$G$-tangle whose $G$-closure is a knot $K$, then $$\e_{\DDP}(\ZHrhoc(K_o))\in\FF$$
depends on $\rho$ only up to conjugation hence it is an invariant of $(K,\rho)$, see \cite{LN:TDD}. Here $\e_{\DDP}:\DDP_{\a\ot\tt}\to\FF$ is given by $\e(p\ot x)=p(1_H)\e_H(x)$ for each $p\in H^*,\ x\in H$. The map $\e_{\DDP}$ is not an algebra morphism over $\DDP_{\a\ot\tt}$ if $\a\neq\id$ so there is no reason for the above evaluation to be trivial. Note that $\ZHrhoc(K_o)$ is itself an invariant of $(K,\rho)$ when $\rho$ is abelian \cite{LN:TDD}, in particular $Z_{\underline{\DDP}}^{\tt}(K_o)$ is an invariant of $K$. If $w(K_o)$ is the writhe of $K_o$ then $$P_H^{\rhoc}(K)=t^{w(K_o)|\La_l|/2}\e_{\DDP}(\ZHrhoc(K_o))$$
is an $\FF$-valued invariant of the unframed $G$-knot $(K,\rho)$. Here $t^{|\La_l|}$ is shorthand for $t_1^{|\La_l|_1}\dots t_m^{|\La_l|_m}$ where $(|\La_l|_1,\dots,|\La_l|_m)$ is the $\N^m$-degree of the cointegral $\La_l$ of $H$. When $\rho\equiv 1$ this simply becomes
\begin{align}
\label{eq: knot poly abelian case}
    P_H^{\tt}(K)=t^{w(K_o)|\La_l|/2}\e_{\DDP}(Z_{\uDHp}^{\tt}(K_o))
\end{align}

\noindent and this is a polynomial invariant of the unframed, oriented knot $K$ (without any additional structure). It is easy to show that this polynomial actually belongs to $\kk[t_1^{\pm 1},\dots,t_m^{\pm 1}]\sb \FF$.

\subsection{Elementary properties}\label{subs: lemmas} %Properties: doubling lemma and orientation reversal in the graded case. In particular, we need the invariant to be defined for all tangles (with no closed components), not just $(1,1)$-tangles. See \cite{Habiro:bottom-tangles}.

In the untwisted case, the universal invariant of the composition of two tangles is obtained by multiplying the beads of the components being composed, doubling a strand corresponds to applying the coproduct, and reversing the orientation corresponds to applying the antipode (see e.g. \cite{Habiro:bottom-tangles}). These properties extend to our invariant $\ZArho$, except that we have to take care of the $\v_{\a}$'s.

\medskip
To simplify the statement, we only state the composition rule in the following special case. Suppose $T$ is a $G$-tangle having two components $T_i,T_{i+1}$ with two endpoints in $\partial_1 X$ that are adjacent and labelled by the same $\a\in G$, see Figure \ref{fig:comp}. We suppose that, at such endpoints, $T_i$ is oriented downwards and $T_{i+1}$ upwards and the endpoint of $T_i$ is to the left of that of $T_{i+1}$. Thus, we can compose $T$ with an $\a$-labelled left cap on top. Let $(T',\rho')$ be the resulting $G$-tangle.

\begin{figure}[htp!]
    \centering
    \includegraphics[width=6cm]{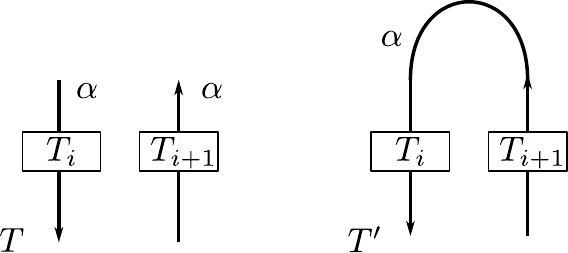}
    \caption{$T'$ is the composition of strands $T_i$ and $T_{i+1}$.}
    \label{fig:comp}
\end{figure}

\begin{lemma}
\label{lemma: COMPOSITION}
The invariant $Z_{\uDH}^{\rho'}(T')$ is computed by $$Z_{\uDH}^{\rho'}(T')=(\id\ot (m_{\c}\circ(\id\ot\v_{\rho([T_i])}))\ot \id)(\ZArho(T))$$
where $m_{\c}$ is applied on the factors corresponding to $T_i$ and $T_{i+1}$ and $\c$ is the $G$-label of the endpoint of $T_i$.
%The universal invariant of $(T_2,\rho_2)\circ (T_1,\rho_1)$ is given by $$Z^{\rho_1\ast\rho_2}_{\uDH}(T_2\circ T_1)=Z^{\rho_2}_{\uDH}(T_2)\cdot \v_{[T_2]}(Z^{\rho_1}_{\uDH}(T_1)).$$
\end{lemma}
\begin{proof}
With the above conventions, the new component in $T'$ consists of $T_{i+1}$, with its orientation, followed by $T_i$. The bead of this component is obtained by sliding the bead on $T_{i+1}$ along $T_i$, evaluating on all $\v_{\b}$'s encountered along $T_i$, and multiplying with the bead of $T_i$. Since the product of the $\v_{\b}$'s along $T_i$ equals $\v_{\rho([T_i])}$, this proves the lemma.
 \end{proof}

%Note that if we were working with $\ZHrho$, then we would need to multiply the beads of $T_i$ and $T_{i+1}$ in $\uA$, say $x_i\rho([T_i])$ and $x_{i+1}\rho([T_{i+1}])$, and then look at their component in $\uDH$. Since $x_i\rho([T_i])x_{i+1}\rho([T_{i+1}])=x_i\v_{\rho([T_i])}(x_{i+1})\rho([T_iT_{i+1}])$, the composition property in $\uDH$ involves an application of $\v$ to one of the factors. This nuisance is the reason we introduced $\uA$.

Now let $T_0$ be a component of a $G$-tangle $(T,\rho)$. Let $T'$ be the tangle obtained by doubling $T_0$ along the framing. Suppose $T'$ is a $G$-tangle with $\rho':\pi_1(X_{T'})\to G$ and let $\rho=\rho'|_{\pi_1(X_T)}$. Let $\rho' (\mu_1)=\a, \rho'(\mu_2)=\b$ where $\mu_1,\mu_2$ are the meridians of the two copies of $T_0$ in $T'$ (say, at the endpoints of the copies of $T_0$). Then $\rho(\mu)=\a\b$ where $\mu$ is the meridian corresponding to $T_0$.

\begin{figure}[H]
 \includegraphics[width=7cm]{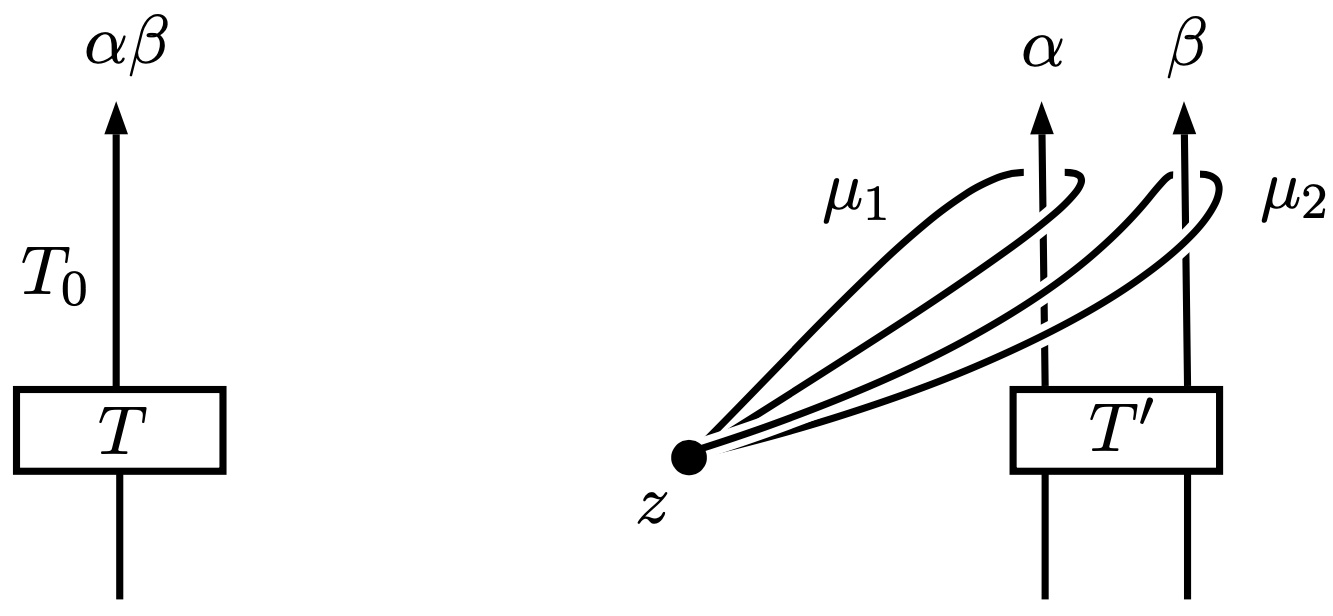}
\caption{On the left is the $G$-tangle $(T,\rho)$. For simplicity, we just depict the component $T_0\sb T$, but $T$ may have any number of components. On the right is $(T',\rho')$ and the meridians $\mu_1,\mu_2$, which have linking number -1 with the corresponding strands of $T'$.}
    
\end{figure}

\begin{lemma}
\label{lemma: DOUBLING}
If $T'$ is constructed by duplicating $T_0\sb T$ as above, then $$(\id\ot\dots\ot \De_{\a,\b}\ot\dots\ot\id)(\ZArho(T))=Z_{\uDH}^{\rho'}(T')$$
where $\De_{\a,\b}$ is applied on the tensor factor corresponding to $T_0$.
\end{lemma}
\begin{proof}
We need to check this on the elementary tangles. For the crossing tangles, this follows by the defining property of the graded $R$-matrix. For right caps and cups, it follows from $\De_{\a,\b}(\gg_{\a\b})=\gg_{\a}\ot \gg_{\b}$ and for left caps/cups it is obvious. The fact that the $\v_{\a}$'s are Hopf isomorphisms implies that the property holds for compositions of elementary tangles, thus for all tangles.
\end{proof}

\begin{lemma}
\label{lemma: REVERSE ORTTN}
Let $T_0\sb T$ be a component oriented upwards, that is, it begins on $\partial_0 X$ and ends on $\partial_1 X$. Let $\d$ be the label at the endpoint of $T_0$. Let $T'$ be the tangle obtained by reversing the orientation of $T_0$. If we use the same $\rho$ for both tangles, then
$$(\id\ot\dots \ot \v_{\rho([T_0])}^{-1}\circ S_{\d}\ot\dots\ot \id)(\ZArho(T))=\ZArho(T')$$
where $\v^{-1}_{\rho([T_0])}\circ S_{\d}$ is applied on the tensor factor corresponding to $T_0$.
\end{lemma}

\begin{remark}
\label{remark: reverse OR}
Recall that in our definition of $\ZHrho$ we ``gather" the beads at the endpoint of each component and then multiply. If we reverse the orientation of $T_0$ and we gather the beads at the same point as before (which is the starting point of $T'_0$), then we just need to apply the antipode $S_{\d}$, but if we want to gather at the endpoint of $T'_0$ we need to slide our bead down through all of $T_0$, this gives the additional $\v_{\rho([T_0])}^{-1}$ above.

\end{remark}

%{\color{blue} Think about the usual case (untwisted): If a component is oriented downwards and you re-orient upwards, then you should apply $S^{-1}$ to the former to get the latter. A cmpt oriented downwards colored by $V$ represents a map $V^*\to V^*$ but if you re-orient and keep the color, it is a map $V\to V$ and not $V^{**}\to V^{**}$ (this shld be a rmk).}

\begin{proof}

Suppose first we reverse the understrand of a positive crossing (oriented upwards) whose bottom labels are $\a,\b$. Thus, $\d=\a\b\a^{-1}$ and $[T_0]=\a$. Then the invariant is computed as on the left of the following figure:
\begin{figure}[H]
 \includegraphics[width=8cm]{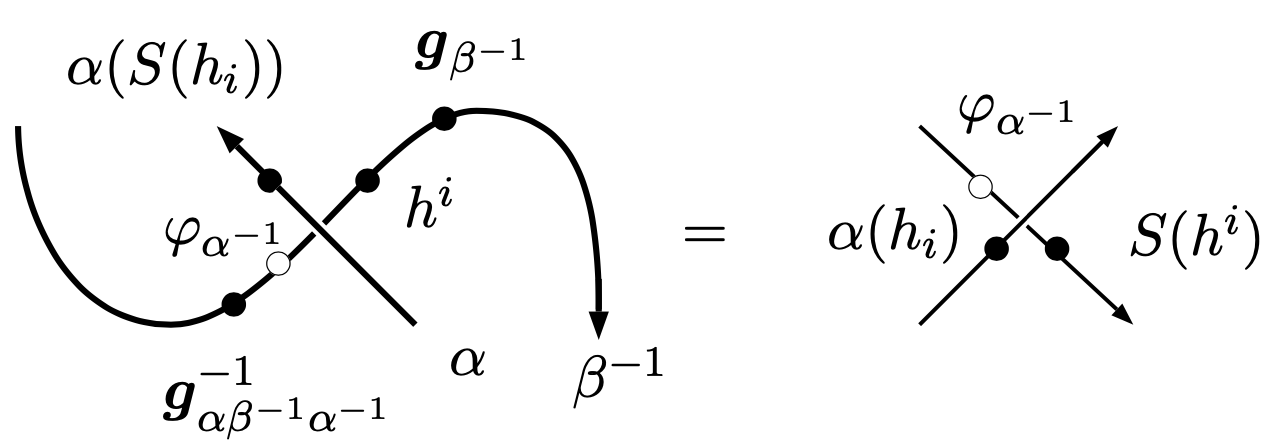}
\end{figure}
For convenience we keep track of the white beads in our pictures. That this equals the picture on the right follows from the following computation:
\begin{align*}
    \a(S(h_i))\ot \gg_{\b^{-1}}h^i\v_{\a^{-1}}(\gg^{-1}_{\a\b^{-1}\a^{-1}})&=\a(S(h_i))\ot \gg_{\b^{-1}}h^i\gg^{-1}_{\b^{-1}}\\
    &=\a(S(h_i))\ot S_{\b}S_{\b^{-1}}(h^i)\\
    &=\a(S(h_i))\ot S^2(h^i)\\
    &=\a(h_i)\ot S(h^i).
\end{align*}
But the picture on the right is exactly obtained by applying $\v_{\a}^{-1}S_{\d}$ to the strand that is being reversed (note that $S_{\d}$ is simply $S_{H^*}$ on $H^*$). Similarly, if we reverse the overstrand of a positive crossing, the invariant is computed as on the left of the following figure:
\begin{figure}[H]
    \centering
  \includegraphics[width=8cm]{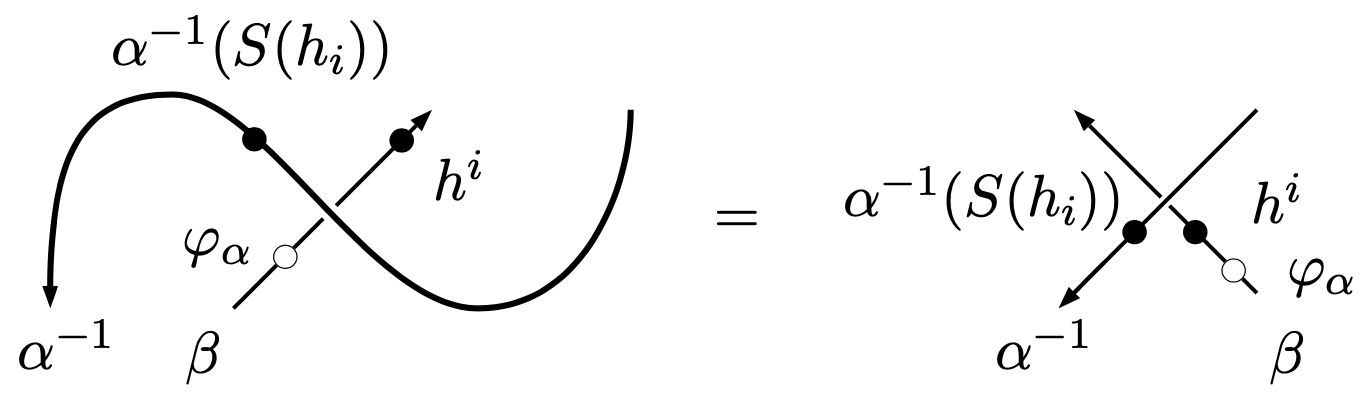}
\end{figure}
Here we have $\d=\a$ and $[T_0]=1$. It is easy to see that the right hand side is obtained by applying $\v_{1}^{-1}S_{\a}=S_{\a}$ to the tangle with the strand oriented upwards. For negative crossings, a similar computation shows that the lemma holds as well. Using the above composition property (and that the antipode is an algebra antiautomorphism), we deduce the lemma for in the case $T_0$ that has no caps or cups.

Now, suppose $T_0$ has caps and cups. We will subdivide (a diagram of) $T_0$ in subarcs $a_1,\dots,a_k$ without caps/cups. Thus, as we follow the orientation of $T_0$, we follow first $a_1$, then a cap, then $a_2$, then a cup, and so on until $a_k$. Each $a_i$ with odd $i$ is oriented upwards and for even $i$ it is oriented downwards. We let $\a_i$ be the label at the starting point of $a_i$, thus, $\a_{i+1}$ is the label of the endpoint of each $a_i$ and we let $\a=\a_1$ and $\a_{k+1}=\d$. We also let $\c_i$ be the image of the partial longitude from the endpoint of $T_0$ to the top of $a_i$, and $\c'_i$ be the partial longitude for $T'_0$, that is, $\c'_i$ goes from the starting point of $T_0$ to the top of $a_i$. We also denote $\c=[T_0]$, so $\c_i(\c'_i)^{-1}=\c$ for each $i$.  For each $i$ (even or odd) we let $x_i$ be the total bead for that arc as if the arc was oriented upwards (and the bead is gathered at the top so $x_i\in \DD_{\a_{i+1}}$ for odd $i$, $x_i\in \DD_{\a_i^{-1}}$ for even $i$). Thus, by what was shown above, for even $i$ the actual bead along that arc is $S_{\a_i^{-1}}(x_i)$ (if the bead is gathered at the top, see Remark \ref{remark: reverse OR} above). Then, the $T_0$ tensor factor of $\ZArho(T)$ can be written as $$ x_k\gg_{\b}^{\e_k}\dots \gg_{\b}^{\e_3}\v_{\c_2}(S_{\a_2^{-1}}(x_2))\gg_{\b}^{\e_2}\v_{\c_1}(x_1)$$
where $\e_i$ is the power of the pivot at the cap/cup right before $a_i$ (so $\e_i\in \{0,\pm 1\}$). Now we apply $\v_{\c^{-1}}S_{\d}=S_{\a}\v_{\c^{-1}}$. First we apply $\v_{\c^{-1}}$ this results in
$$ \v_{\c'_k}(x_k)\gg_{\a}^{\e_k}\dots \gg_{\a}^{\e_3}\v_{\c'_2}(S_{\a_2^{-1}}(x_2))\gg_{\a}^{\e_2}\v_{\c'_1}(x_1).$$
Applying $S_{\a}$ gives
\begin{align*}
   & S_{\a}(\v_{\c'_k}(x_k)\gg_{\a}^{\e_k}\dots \gg_{\a}^{\e_3}\v_{\c'_2}(S_{\a_2^{-1}}(x_2))\gg_{\a}^{\e_2}\v_{\c'_1}(x_1))\\
   &=S_{\a}\v_{\c'_1}(x_1)\gg_{\a^{-1}}^{-\e_2}S_{\a}\v_{\c'_2}S_{\a_2^{-1}}(x_2)\gg_{\a^{-1}}^{-\e_3}\dots \gg_{\a^{-1}}^{-\e_k}\v_{\c'_k}S_{\a}\v_{\c'_k}(x_k)\\
  &= \v_{\c'_1}(S_{\a_2}(x_1))\gg_{\a^{-1}}^{-\e_2}\v_{\c'_2}(S_{\a_2}S_{\a_2^{-1}}(x_2))\gg_{\a^{-1}}^{-\e_3}\dots \gg_{\a^{-1}}^{-\e_k}\v_{\c'_k}S_{\d}(x_k)\\
  &=\v_{\c'_1}(S_{\a_2}(x_1))\gg_{\a^{-1}}^{-\e_2}\v_{\c'_2}(\gg_{\a_2^{-1}}x_2\gg_{\a_2^{-1}}^{-1})\gg_{\a^{-1}}^{-\e_3}\dots \gg_{\a^{-1}}^{-\e_k}\v_{\c'_k}S_{\d}(x_k)\\
  &=\v_{\c'_1}(S_{\a_2}(x_1))\gg_{\a^{-1}}^{-\e_2+1}\v_{\c'_2}(x_2)\gg_{\a^{-1}}^{-1-\e_3}\dots \gg_{\a^{-1}}^{-\e_k}\v_{\c'_k}S_{\d}(x_k)
\end{align*}

%When we reverse the orientation of $T_0$, the invariant becomes $$S(x_1)g^{\e'_1}x_2g^{\e'_2}\dots=S(\dots g^{-\e'_2}S^{-1}(x_2)g^{\e'_1}x_1)$$
We used that $S_{\b}S_{\b^{-1}}(x)=\gg_{\b^{-1}}x\gg_{\b^{-1}}^{-1}$ for each $x\in \DD_{\b^{-1}}$ and $\b$ and that $S_{\a}\v_{\c'_i}=\v_{\c'_i}S_{(\c'_i)^{-1}\a\c'_i}$ and $(\c'_i)^{-1}\a\c'_i=\a_{i+1}$ for odd $i$ or $=\a_i$ for even $i$. Note that $\e'_i+\e_i=1$ for each even $i$ (which correspond to caps) and $\e'_i+\e_i=-1$ for each odd $i$ (corresponding to cups), where $\e'_i$ is the power of the pivot at the given cap/cup when the orientation is reversed. Therefore, the above equals $$\v_{\c'_1}(S_{\a_2}(x_1))\gg_{\a^{-1}}^{\e_2'}\v_{\c'_2}(x_2)\gg_{\a^{-1}}^{\e_3'}\dots \gg_{\a^{-1}}^{\e_k'}\v_{\c'_k}S_{\d}(x_k)$$
which is exactly the bead of $T_0$ with its inverse orientation.
\end{proof}

Now let $T_0$ be a component of a tangle $T$ as above, but suppose now that both endpoints of $T_0$ lie on $\partial_1X$ and that the endpoint of $T_0$ is to the right of its starting point. Let $\wt{S}_{\a}:\DD_{\a}\to \DD_{\a^{-1}}$ be defined by
\begin{align}
\label{eq: antipode twisted by pivot}
    \wt{S}_{\a}(x)=S_{\a}(x)\gg^{-1}_{\a^{-1}}.
\end{align}
Let $\d$ be the label of the endpoint of $T_0$. If $T'$ is obtained by reversing the orientation of $T_0$ then the above lemma implies that 
\begin{align}
\label{eq: reverse OR antipode twisted by pivot}
   (\id\ot\dots \ot \v^{-1}_{\rho([T_0])}\wt{S}_{\d}\ot\dots\ot \id)(\ZArho(T))=\ZArho(T').
\end{align}

\medskip

When we use the abelian representation $\tt_T=\tt$ one has to be a bit more careful. This representation depends on the orientation of each component, in particular $\tt_{T'}$ is not equal to $\tt_{T}$. The beads at a negative crossing (of the first type) are given as follows:

\begin{figure}[H]
 \includegraphics[width=7cm]{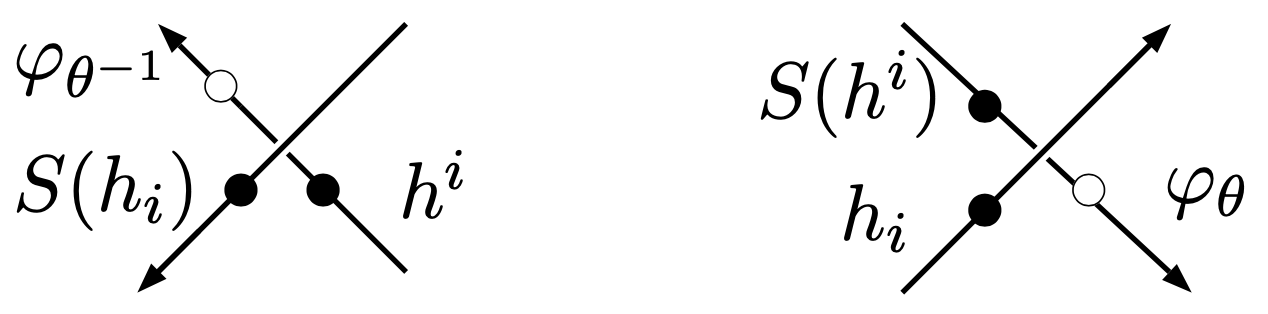}
\end{figure}

\def\uDHp{\underline{D(H')}}
\subsection{Twist knots}\label{subs: twist knots} 
To illustrate our constructions we will compute the twisted universal invariant explicitly on the family of twist knots. 
Let $K_n$ be the $2n$-twist knot. This is an alternating knot with $2n+2$ crossings and genus 1. The knot $K_1$ is the figure-eight knot, $K_2$ (Stevedore's knot) is drawn on the left of Figure \ref{fig: twist knot}. For simplicity, we suppose $m=1$, that is, $H$ is an $\N$-graded Hopf algebra. Recall that the pivot of $D(H')_{\tt}$ is $t^{-d(H)/2}\gg$ where $\gg=\bbeta\ot\bb$ is the pivot of $D(H)$ (note that $d(H)=|\La_l|$ for $m=1$). We also take $\rho\equiv 1$ (but we keep $\tt:H_1(X_K)\to\Aut(H)$). Thus all white beads in our diagrams will be $\v_{\tt^{\pm 1}}$, we will denote these simply by $\tt^{\pm 1}$. We will show directly that $\deg \ P_H^{\tt}(K_n)\leq 2d(H)$ for any $n$. This is in agreement with the genus bound, Theorem \ref{Thm: main thm}, since as mentioned the genus of every twist knot is $1$.
In what follows we denote $Z^{\tt}=Z_{\uDHp}^{\tt}(K_n)$. 
\medskip

\begin{figure}[t]
  \includegraphics[width=12cm]{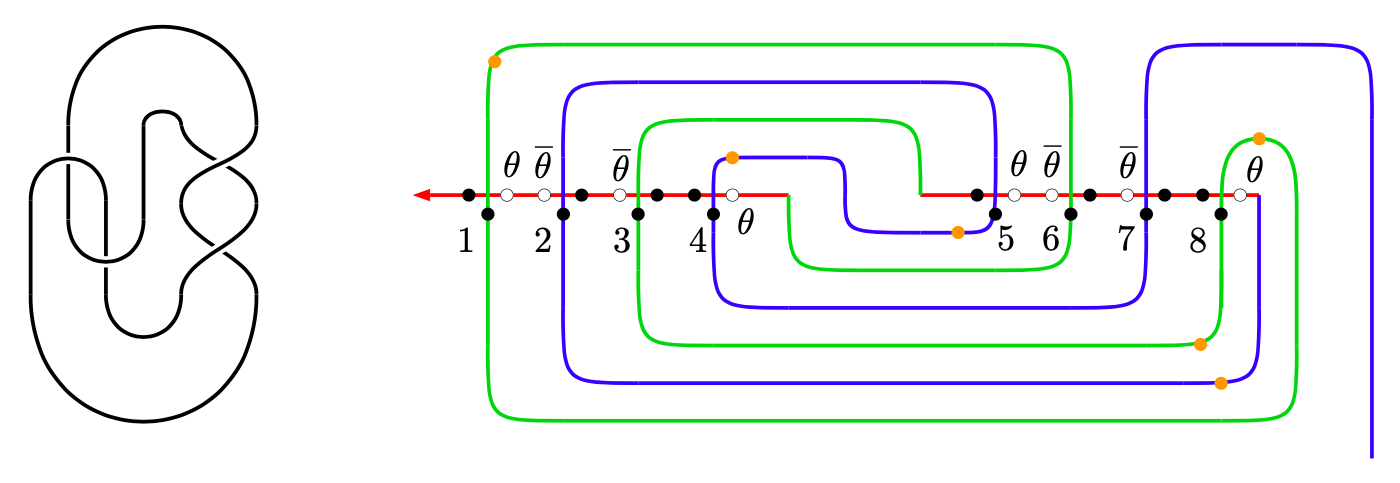}

    \caption{The figure eight knot and a 2-bridge presentation (opened to be a $(1,1)$-tangle). The black dots represent elements of $H,H^*$, the white dots represent $\tt^{\pm 1}$'s and the orange dots on right caps (resp. right cups) represent $\gg=t^{-d(H)/2}\bbeta\ot\bb$ (resp. $\gg^{-1}$). Here $\ov{\tt}$ denotes $\tt^{-1}$.}
    \label{fig:figure eight}
\end{figure}

Let's begin with the figure eight knot. We compute $Z^{\tt}$ from the 2-bridge diagram of the right of Figure \ref{fig:figure eight}. We use bridge presentations with minimal number of bridges to reduce the number of Drinfeld products to use, this simplifies considerably the expression for $Z^{\tt}$. Note that this diagram has rotation number zero so all the powers of $t$ from pivots cancel out, in other words, we can simply put the untwisted pivot $\gg=\bbb\ot\bb$ on right caps and $\gg^{-1}$ on right cups, which is what we do. The diagram also has writhe zero so that $\e(Z^{\tt})=P_H^{\tt}(K_1)$. In a bridge diagram, other from the $\gg^{\pm 1}$'s on caps/cups, all the beads over an overarc are in $H$ and all beads on an underarc are elements of $H^*$. All white beads, which are $\tt^{\pm 1}$, lie on underarcs. We number the crossings from left to right as shown in Figure \ref{fig:figure eight}. If $\sum h_j\ot h^j$ represents the $R$-matrix (of $D(H)$) on the $j$-th crossing, then the invariant is given by (the using Einstein summation convention to hide the eightfold sum):
\begin{align*}
    Z_{\uDHp}^{\tt}(K_1)&= h^1h^2\tt^{-1}(h^3h^4)\cdot_{\tt} S(h_6)\gg h_1\gg h_8\gg^{-1}S(h_3)\cdot_{\tt} h^5h^6\tt^{-1}(h^7h^8)\cdot_{\tt} \gg^{-1}S(h_2)h_5h_4S(h_7)
    %&=(h^1h^2\tt^{-1}(h^3h^4)\ot S(h_6)gh_1gh_8g^{-1}h_3)\cdot_{\tt} (h^5h^6\tt^{-1}(h^7h^8)\ot g^{-1}S(h_2)h_5h_4S(h_7))
    %&=(h^1h^2\tt^{-1}(h^3h^4)\bbeta\ot\bbS^{-1}(h_6)h_1S^2(h_8)h_3)\cdot_{\tt} (h^5h^6\tt^{-1}(h^7h^8)\b^{-1}\ot b^{-1}S(h_2)h_5h_4S(h_7))
\end{align*}
Here $\cdot_{\tt}$ denotes the multiplication of the twisted double $D(H')_{\tt}$.
%Note that $Z_{\uA}=Z_{\uDH}=Z$ in this case. In the second equality we used the relation $\tt\cdot x=\tt(x)\cdot\tt$ in $\uA$. % and in the third we used $p\cdot_{\tt}a=p\ot a$ for $p\in H^*, a\in H$. 
Note that $$\e_{D(H')}(p\cdot_{\tt}x\cdot_{\tt} a)=\e_{H^*}(p)\e_{H}(a)\e_{D(H')}(x)$$
for any $p\in H^*, a\in H, x\in D(H')_{\tt}$.
Thus $\e(Z^{\tt})$ equals
\begin{align*}
   &\e(h^1h^2\tt^{-1}(h^3h^4))\e( \bb^{-1}S(h_2)h_5h_4S(h_7))\e[(S(h_6)\gg h_1\gg h_8\gg^{-1}S(h_3))\cdot_{\tt} (h^5h^6\tt^{-1}(h^7h^8)\bbb^{-1})]\\
    &=\e[S(h_6)\gg^2h_8\gg^{-1}\cdot_{\tt}h^6\tt^{-1}(h^8)\bbb^{-1}].
\end{align*}

Here we used that $\e$ is an algebra morphism over $H$ and over $H^*$ (though not over all of $D(H')_{\tt}$) and that $\e(h_j)h^j=h_j\e(h^j)=1$ for each $j$. Note that $$\e(x'\cdot_{\tt}y')=\lb S^{-1}\tt^{-1}(x'_{(2)}),y'_{(1)}\rb\lb x'_{(1)},y'_{(2)}\rb$$
for each $x'\in H, y'\in H^*$. Thus, writing $S(h_6)\gg^2h_8\gg^{-1}=\bbb x'$ with $x'\in H$ and $y'=h^6h^8\bbb^{-1}\in H^*$ ($x',y'$ have no meaning independently, but $x'\ot y'$ does) and using $\tt^{-1}(h^i)=t^{-|h^i|}h^i=t^{|h_i|}h^i$ we obtain

\begin{align*}
    \e(Z^{\tt})&=\e(\bbb x'\cdot_{\tt}t^{|h_8|}y')\\
    &=t^{|h_8|}\e(x'\cdot_{\tt}y')\\
    &=t^{|h_8|}\lb S^{-1}\tt^{-1}(x'_{(2)}),y'_{(1)}\rb\lb x'_{(1)},y'_{(2)}\rb\\
    &=t^{|h_8|-|x'_{(2)}|}\lb S^{-1}(x'_{(2)}),y'_{(1)}\rb\lb x'_{(1)},y'_{(2)}\rb.   %\in \sum t^{|h_8|-|x'_{(2)}|}\kk\sb \kk[t^{\pm 1}].
\end{align*}
Note that $\lb S^{-1}(x'_{(2)}),y'_{(1)}\rb\lb x'_{(1)},y'_{(2)}\rb\in\kk$, so the maximal/minimal power of $t$ in $\e(Z^{\tt})$ is the maximum/minimum of $|h_8|-|x'_{(2)}|$. But since $0\leq |h_8|\leq d(H),0\leq |x'_{(2)}|\leq d(H)$ (this uses that $H$ is $\N$-graded, not simply $\Z$-graded) it follows that $-d(H)\leq |h_8|-|x'_{(2)}|\leq d(H)$, hence $\e(Z^{\tt})$ is a polynomial of degree $\leq 2d(H)$.
\medskip

We now consider the case of general $K_n$. The knot $K_n$ has a 2-bridge presentation as shown on the right of Figure \ref{fig: twist knot}, but with $n-1$ full rotations of the ``middle band". This diagram has rotation number and writhe zero as before %so as before $Z=Z_{\uDH}^{\tt}(K_n)=Z_{\uA}^{\tt}(K_n)$ 
so we can simply use $\gg^{\pm 1}$ on right caps/cups. Let $b_1,a_1,b_2,a_2$ be the arcs of the bridge presentation as one follows the orientation of $K$. In Figure \ref{fig: twist knot}, $b_1$ is blue and $b_2$ is green. As in the example of the figure-eight knot, the product of all elements along $a_1$ (resp. $a_2$) is an element $q\in H^*$ (resp. $p\in H^*$). The beads over $b_1,b_2$ are all in $H$, except for the pivots $\gg=\bbeta\ot\bb$ (or inverses). As before, for each arc, we multiply the beads encountered along that arc and gather the product at the very end of that arc. The products along $b_2,b_1$ have the form $\bbb c, \bbb^{-1}d$ for some $c,d\in H$. Then the invariant is $Z^{\tt}=p\cdot_{\tt} \bbb c\cdot_{\tt} q\cdot_{\tt} \bbb^{-1}d$ and $$\e_{D(H')}(Z^{\tt})=\e( p\bbb)\e(d)\e(c\cdot_{\tt}q).$$

As before, $h_j\e(h^j)=\e(h_j)h^j=1_{D(H)}$ so that evaluating $p,d$ on $\e_{D(H')}$ has the effect of killing all the $D(H)$-beads coming from $a_2$ or $b_1$. The only beads that remain are those of the crossings between $a_1$ and $b_2$, the $\tt^{\pm 1}$'s on $a_1$ and the pivots $\gg^{\pm 1}$'s on $b_2$. In other words, except for the pivots, all the remaining beads are shown below (we denote $\ov{\tt}=\tt^{-1}$):
\begin{figure}[H]
 \includegraphics[width=10cm]{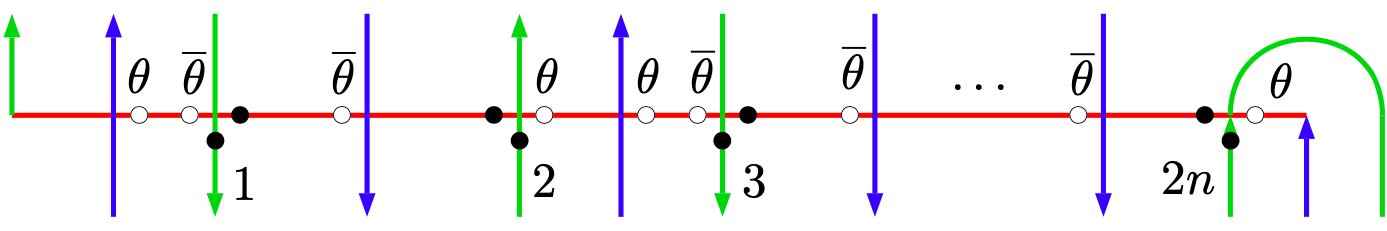}
\end{figure}

Order the crossings between $a_1$ and $b_2$ by $1,2,\dots,2n$ from left to right. Let $h_j\ot h^j$ be the $D(H)$-component of the $R$-matrix on the $j$-crossing. Then, the product of all beads along the red arc is
\begin{align*}
    y&=h^1\tt^{-1}(h^2)h^3\tt^{-1}(h^4)\dots \tt^{-1}(h^{2n})\\
    &=t^{|h_2|+\dots+|h_{2n}|}y'
\end{align*}
where $y'=h^1h^2\dots h^{2n}\in H^*$.
The product along $b_2$ is
\begin{align*}
    x&=S(h_1)\gg\dots S(h_{2n-1})\gg\cdot\gg \cdot h_{2n}\gg^{-1}\dots h_4\gg^{-1}h_2\gg^{-1}.
\end{align*}
Since $S^2(a)=\gg a\gg^{-1}$ for any $a\in H$, this has the form $x=\bbb x'$ with $x'\in H$. Thus
\begin{align*}
    \e(Z^{\tt})&=\e(x\cdot_{\tt}y)=\e(x'\cdot_{\tt}y)\\
    &=\lb S^{-1}\tt^{-1}(x'_{(2)}),y_{(1)}\rb\lb x'_{(1)},y_{(2)}\rb\\
    &=t^{|h_2|+|h_4|+\dots+|h_{2n}|-|x'_{(2)}|}\lb S^{-1}(x'_{(2)}),y'_{(1)}\rb\lb x'_{(1)},y'_{(2)}\rb
\end{align*}

 The term $\lb S^{-1}(x'_{(2)}),y'_{(1)}\rb\lb x'_{(1)},y'_{(2)}\rb\in\kk$. Now, since $H$ is $\N$-graded the product $h^1\dots h^{2n}$ is zero if $|h_1|+\dots+|h_{2n}|>d(H)$, so the only terms that contribute to the above sum are those for which $0\leq |h_2|+\dots +|h_{2n}|\leq d(H)$. Since also $0\leq |x'_{(2)}|\leq d(H)$, it follows that the degree of $\e(Z^{\tt})$ is $\leq 2d(H)$ as we wanted.

\begin{figure}[t]
  \includegraphics[width=12cm]{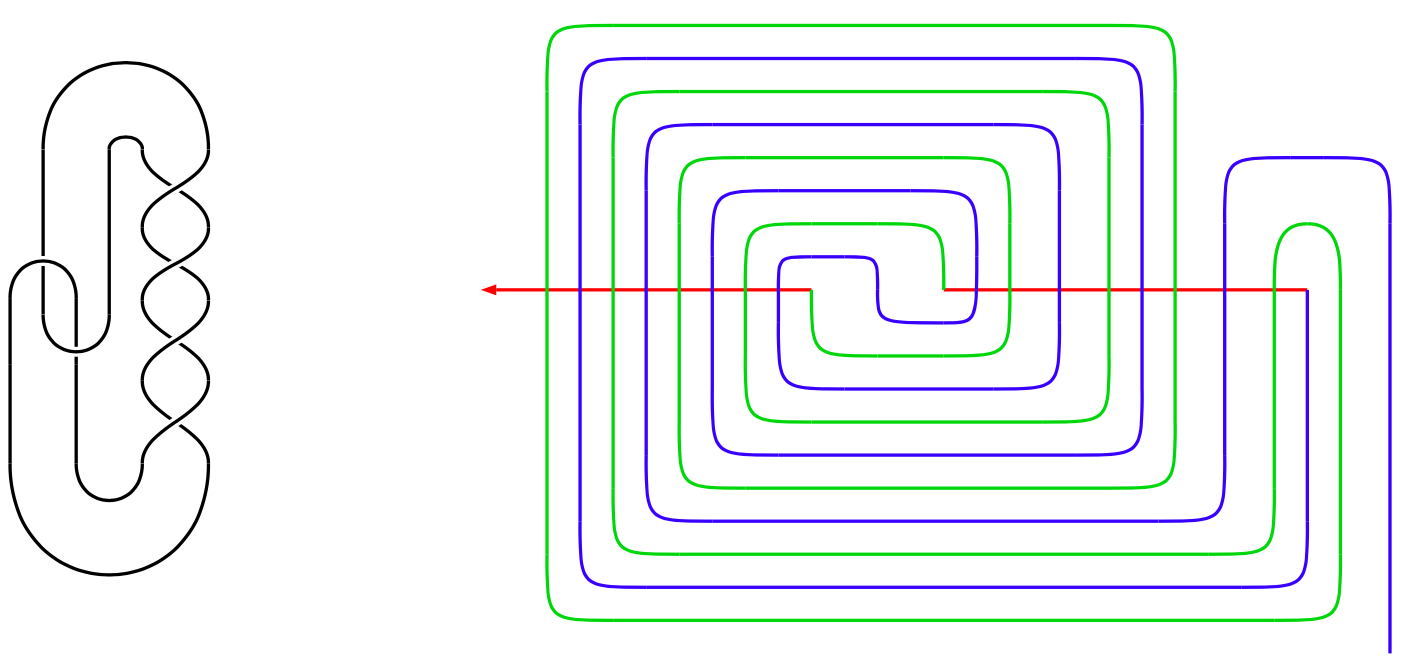}

  \caption{On the left is the twist knot $K_2$ and on the right, a bridge presentation of $K_2$, opened to be a $(1,1)$-tangle. The horizontal red arcs are supposed to be the underarcs.}
     \label{fig: twist knot}
\end{figure}

\section{Proof of Theorem \ref{Thm: main thm}}\label{sect: PROOF}

\def\ait{\a_i\ot\theta}\def\aitm{\a_i\ot\theta^{-1}}\def\bit{\b_i\ot\theta}\def\bitm{\b_i\ot\theta^{-1}}

\def\at{\a\ot\theta}\def\atm{\a\ot\theta^{-1}}\def\bt{\b\ot\theta}\def\btm{\b\ot\theta^{-1}}\def\RTrho{RT_{\uDH}^{\rho}}\def\RTrhoc{RT_{\uDH}^{\rhoc}}\def\tt{\theta}

\def\DDP{\DDp}

Let $H$ be a finite dimensional $\N^m$-graded Hopf algebra over $\kk$ and let $H'=H\ot_{\kk}\FF$ where $\FF=\kk[t_1^{\pm 1/2},\dots,t_m^{\pm 1/2}]$. In what follows we denote $G=\Auto_{\N^m}(H)\sb \Aut(H)$ which we see as a subgroup of $G'\sb \Aut(H')$ via $\a\mapsto \a\ot \tt^0$ and we denote $\DDP_{\a\ot\tt^0}$ simply by $\DDP_{\a}$, that is, $\DDP_{\a}=\DD_{\a}\ot\FF$ for $\a\in G$. The reader interested on the case of ADO may simply suppose $m=1$ and $G=1$.

\def\uDHp{\underline{D(H')}}
\subsection{Formulas from Seifert surfaces}   Let $S$ be a Seifert surface for $K$ and let $g=g(S)$. We suppose $K$ is 0-framed, this is enough for our theorem. After an isotopy, we can suppose $S$ is obtained by thickening a $2g$-component (framed) tangle $T$ all of whose endpoints lie on $\partial_1 X$ and attaching a disk to the top as in Figure \ref{fig: Seifert surface}. We suppose this disk comes with a marked point $p$ on its boundary and let $K_o$ be the $(1,1)$-tangle obtained by opening $K=\p S$ at that point. We denote the components of $T$ by $T_1,\dots,T_{2g}$ from left to right and we suppose that each $T_i$ is oriented ``to the right", that is, the orientation at the rightmost endpoint is upward pointing. Then, $K_o$ is obtained by first doubling all the components of $T$ (with their orientations), reversing the orientation of the right parallel of each component and then composing with various left caps on the top. In the untwisted case, this implies that the universal invariant $Z_{D(H)}(K)=Z_{D(H)}(K_o)\in D(H)$ is obtained from $Z_{D(H)}(T)\in D(H)^{\ot 2g}$ by applying the coproduct to each tensor factor, applying antipodes on the even tensor factors (with a pivot) and finally multiplying:
\begin{align}
\label{eq: untwisted Z from SS}
    Z_{D(H)}(K_o)=m_{D(H)}^{(4g-1)}\circ P^{\ot g}\circ(\id_{D(H)}\ot \wt{S}_{D(H)})^{\ot 2g}\circ\De_{D(H)}^{\ot 2g}\left(Z_{D(H)}(T)\right)
\end{align}
where $P$ is the permutation map $P(u\ot v\ot x\ot y)=v\ot x\ot u\ot y$ (with signs in the super case) and $\wt{S}_{D(H)}(x)=S_{D(H)}(x)\gg^{-1}$. The additional pivot in the antipode comes from the fact that we reversed the orientation of a component of a tangle all of whose endpoints lie on $\R^2\t 1$. This idea comes back to Habiro \cite{Habiro:bottom-tangles}. However, what really makes our theorem work is the appearance of the representation $\rhoc$, or rather, the degree twist $\tt$. The essential observation is that, when $\rhoc$ is restricted to $\pi_1(X_T)$, $\tt$ cancels out, that is $$\rhoc|_{\pi_1(X_T)}=\rho|_{\pi_1(X_T)}.$$
This is because the orientations of $K$ are opposite on each pair of doubled components of $T$. More precisely, let $\mu'_i,\mu''_i$ be the loops corresponding to the rightmost endpoints of the double of a component $T_i$ of $T$ and $\mu_i$ be the loop associated to the rightmost endpoint of $T_i$ itself so that $\mu_i=\mu'_i\mu''_i$. Since $[\mu'_i]=+1=-[\mu''_i]$ in $H_1(X_K)\cong\Z$ we have $\rhoc(\mu'_i)=\a_i\ot\tt$ and $\rhoc(\mu''_i)=\b_i\ot\tt^{-1}$, where $\a_i=\rho(\mu'_i),\b_i=\rho(\mu''_i)$, so that $\rhoc(\mu_i)=\rho(\mu_i)$. The effect of $\tt$ cancelling out is that, in the twisted version of (\ref{eq: untwisted Z from SS}) above, the invariant of $T$ has no powers of $t$ involved. Indeed, since $\rho$ takes values in $G\sb G'$ and the $R$-matrix of $\underline{D(H')}|_G$ is induced from that of $\uDH|_G$ it follows that $$Z_{\uDHp}^{\rho}(T)=Z_{\uDH}^{\rho}(T)\in \bigotimes_{i=1}^{2g} \DD_{\a_i\b_i}\sb \bigotimes_{i=1}^{2g}\DDP_{\a_i\b_i\ot\tt^0}.$$

\begin{figure}
 \includegraphics[width=10cm]{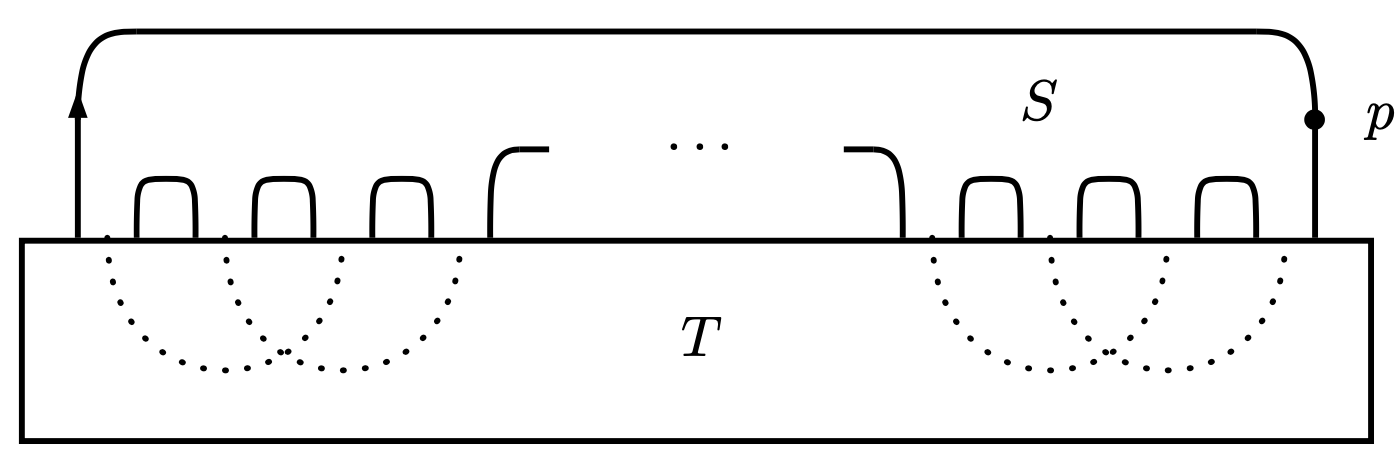}
\caption{$T$ is a tangle that is thickened to give the lower portion of the surface. $T$ is an arbitrary tangle with $2g$ components $T_1,\dots, T_{2g}$ where the endpoints of $T_{2i-1},T_{2i}$ are adjacent as shown above.}
\label{fig: Seifert surface}
\end{figure}

Let $l'_i,l''_i\in \pi_1(X_{K_o})$ be the partial longitudes associated to $\mu'_i,\mu''_i$ above. Let $\c_i=\rho(l'_i)$ and $\d_i=\rho(l''_i)$ so that $\c_i\a_i\c_i^{-1}=\d_i\b_i^{-1}\d_i^{-1}=\b_1^{-1}$ for each $i$ (note that $\b_1^{-1}$ is the label of the endpoint of $K_o$). Note that $l'_i,l''_i$ are zero in $H_1(X_K)$ so that $\rhoc(l'_i)=\c_i$ and $\rhoc(l''_i)=\d_i$.

\begin{proposition}
\label{prop: SS formula}
The twisted universal invariant of $(K_o,\rhoc)$ is computed as 
$$Z_{\uDHp}^{\rhoc}(K_o)=F\left(Z_{\uDH}^{\rho}(T)\right)$$
where 
$$F=m^{(4g-1)}\circ P^{\ot g}\circ\left(\bigotimes_{i=1}^{2g}\v_{\c_i}\ot\v_{\d_i}\right)\left(\bigotimes_{i=1}^{2g}  (\id_{\DDP_{\ait}}\ot \wt{S}_{\bitm})\circ \De_{\ait,\bitm}\right).$$

Here $m^{(4g-1)}$ denotes iterated multiplication of $\DDP_{\b_1^{-1}\ot\tt}$ and $\wt{S}$ is as in (\ref{eq: antipode twisted by pivot}).
%$$m_t^{\ot(4g-1)}[(\id\ot S_{t^{-1}})\De_{t,t^{-1}}]^{\ot 2g}(\RTrho(T)).$$

\end{proposition}
\begin{proof}
This follows immediately from the lemmas of Subsection \ref{subs: lemmas} and the above remarks.
\end{proof}

Since $Z_{\uDH}^{\rho}(T)$ has no powers of $t$, the above formula implies that all powers of $t$ in $Z_{\uDHp}^{\rhoc}(K_o)$ come from the tensor $F$. Note that this tensor only depends on the genus (and the $\a_i,\b_i,\c_i,\d_i\in G$ if $\rho$ is non-trivial). Thus, this already shows that the powers of $t$ are controlled. The next section shows how to keep track of the powers of $t$ coming from $F$.

\medskip
\subsection{Filtration of $\underline{D(H')}$}

In what follows, if $\ovi=(i_1,\dots,i_m), \ovj=(j_1,\dots,j_m)\in\N^m$, we introduce the following notation:
\begin{itemize}
\item $H=\bigoplus_{\ovi\in\N^m} H_{\ovi}$,
    \item $t^{\ovi}=t_1^{i_1}\dots t_m^{i_m}$,
    \item $d(\ovi)=i_1+\dots+i_m$
    \item $I+J=(i_1+j_1,\dots,i_m+j_m)$ and $-\ovi=(-i_1,\dots,-i_m)$,
    \item $I\leq J$ if $i_k\leq j_k$ for each $k=1,\dots,m$,
    \item If $a\in H$ is an homogeneous element, $a\in H_{\ovi}$, we denote $|a|=\ovi$ and $d(a)=d(\ovi)$. Note that $|ab|=|a|+|b|$ and $d(ab)=d(a)+d(b)$ for homogeneous $a,b\in H$. With this notation we have $\tt(a)=t^{|a|}a$.
    \item For every $n\in\N$ we denote $H_n=\oplus_{d(I)=n}H_I$ so that $H=\oplus_{n\in\N}H_n$. 
    \item Elements of $\Auto_{\N^m}(H)$ are denoted by $\a,\b,\c,...$
\end{itemize}

We define a $\kk$-linear $\N$-filtration on each $D(H')_{\a\ot\tt^n}=(H^*\ot H)[t_1^{\pm 1/2},\dots,t_m^{\pm 1/2}]$, where $n\in\Z$, by $$\DDP_{\a\ot\tt^n}[k]:=\bigoplus_{I,J\geq 0 , d(\ovi)+d(\ovj)\leq k}(H^*\ot H_{\ovi})\cdot t^{-\ovj}.$$
We will denote this simply by $D'_{\a\ot\tt^n}[k]$.
In particular, an element of filtration-degree $k$ is a polynomial in $H^*\ot H$ of degree $\leq k$.
%We extend this to a filtration of $A '_{\a\ot\tt^n}$ by $$A' _{\a\ot\tt^n}[k]:=\kk[G]\ltimes\left( \bigoplus_{[\b]\in G/Z(G)} D'_{\b\a\b^{-1}\ot\tt^n}[k]\right).$$
We consider higher tensor products $\otimes_{i=1}^nD'_{\a_i\ot\tt}$ with the tensor product filtration.
\medskip

\begin{lemma}
The multiplication of $\DDP_{\at}$ is filtration-preserving.
\end{lemma}
\begin{proof}
 Let $x\in D'_{\at}[n], y\in D'_{\at}[m]$, say $x=p\ot a\cdot t^{-J}$ and $y=q\ot b\cdot t^{-L}$, where $a,b\in H,\ p,q\in H^*$ and $d(a)+d(J)\leq n, d(b)+d(L)\leq m$. Then
\begin{align*}
x\cdot_{\at}y&=t^{-J-L}\lb a_{(1)},q_{(3)}\rb\lb S^{-1}((\at)^{-1}(a_{(3)})), q_{(1)}\rb p\cdot q_{(2)}\ot a_{(2)}\cdot b\\
&=t^{-J-L-|a_{(3)}|}\lb a_{(1)},q_{(3)}\rb\lb S^{-1}(\a^{-1}(a_{(3)})), q_{(1)}\rb p\cdot q_{(2)}\ot a_{(2)}\cdot b\\
&\in \bigoplus(H^*\ot H_{|a_{(2)}b|})t^{-J-L-|a_{(3)}|}.
\end{align*}
Since $H$ is $\N^m$-graded, we have $d(a_{(2)})+d(a_{(3)})\leq d(a)$. Thus $d(a_{(2)}b)+d(J+L+a_{(3)})\leq d(a)+d(b)+d(J)+d(L)\leq n+m$, showing that the above direct sum is contained in $D'_{\at}[n+m]$ as desired. %Since the filtration only contains elements of the form $\a\cdot x$ with $x\in D_{\b\a\b^{-1}\ot\tt}[k]$ and $\a\in G$, the result for $A'_{\at}$ follows from $\a\cdot x=\a(x)\cdot\a$ together with the fact that $\a$ preserves the $\N^m$-degree. 
\end{proof}

%Note that the last sentence of the above proof shows why the filtration of $A'_{\at}$ only contains automorphisms of $G$ and not automorphisms of the form $\a\ot\tt^n$ with $n\neq 0$.
\medskip

In contrast to the above lemma, the maps $\De_{\at,\btm},S_{\btm}$ are not filtration-preserving. However, we have the following:

\begin{lemma}
The map $(\id_{\DDP_{\at}}\ot S_{\btm})\circ\De_{\at,\btm}:\DDP_{\a\b}\to \DDP_{\at}\ot \DDP_{\b^{-1}\ot \tt}$ is filtration-preserving.
\end{lemma}
\begin{proof}
 Let $x\in D'_{\a\b}[n]$, say, $x=p\ot a\cdot t^{-J}$ with $d(a)+d(J)\leq n$. We have 
\begin{align*}
(\id_{D(H')_{\at}}\ot S_{\btm})&\De_{\at,\btm}(x) \\
&=t^{-J}(\id_{D(H)_{\at}}\ot S_{\btm})(p_{(2)}\ot a_{(1)}\ot p_{(1)}\ot t^{-|a_{(2)}|}\a^{-1}(a_{(2)}))\\
&=t^{-J}p_{(2)}\ot a_{(1)}\ot (t^{-|a_{(2)}|}t^{|\a^{-1}(a_{(2)})|}\b^{-1}S\a^{-1}(a_{(2)})\cdot_{\b^{-1}\ot\theta} p_{(1)}\circ S^{-1})\\
&=(t^{-J}p_{(2)}\ot a_{(1)})\ot( \b^{-1}S\a^{-1}(a_{(2)})\cdot_{\b^{-1}\ot\theta} p_{(1)}\circ S^{-1}).
\end{align*}
Note that in the third equality we used that $\a\in \Aut(H)$ preserves the $\N^m$-degree. By the previous lemma, the rightmost tensor factor of the last equality belongs to $D'_{\b^{-1}\ot\tt}[d(a_{(2)})]$. The left tensor factor clearly belongs to $D'_{\a\ot\tt}[d(a_{(1)})+d(J)]$ so their tensor product belongs to $(D'_{\a\ot\tt}\ot D'_{\b^{-1}\ot\tt})[d(a_{(1)})+d(a_{(2)})+d(J)]$ which is contained in the $n$-th term of the same filtration since $$d(a_{(1)})+d(a_{(2)})+d(J)\leq d(a)+d(J)\leq n.$$

\end{proof}

\subsection{Proof of Theorem \ref{Thm: main thm}} Let $N=d(H)$.
We will prove something more general, namely, that the entire universal invariant $\ZHrhoc(K_o)\in D(H')$ is a polynomial of degree $\leq 2gN$, after identifying $D(H')$ with $(H^*\ot H)[t_1^{\pm 1/2},\dots,t_m^{\pm 1/2}]$. Thus, any evaluation, in particular $P_H^{\rhoc}$, is a polynomial of degree $\leq 2gN$. To begin with, note that since there are no powers of $t$ involved in $Z_{\uDH}^{\rho}(T)$ and $N$ is the top $\N$-degree of $H$, we have $$Z_{\uDH}^{\rho}(T)\in \bigotimes_{i=1}^{2g} (D'_{\a_i\b_i}[N])\sb \left(\bigotimes_{i=1}^{2g}D'_{\a_i\b_i}\right)[2gN].$$ Recall that $\wt{S}_{\b\ot\tt^{-1}}(x)=S_{\b\ot\tt^{-1}}(x)\gg_{\b^{-1}\ot\tt}^{-1}$ and that $\gg_{\b^{-1}\ot\tt}=t^{-|\La_l|/2} \bbeta\ot\bb$. Clearly, multiplication by $\bbeta\ot\bb$ preserves the filtration. Moreover, $\v_{\a}$ is filtration-preserving for any $\a\in G$. Together with the previous lemmas, it follows that the map $$t^{-g|\La_l|}\cdot m^{(4g-1)}\circ P^{\ot g}\circ\left(\bigotimes_{i=1}^{2g}\v_{\c_i}\ot\v_{\d_i}\right)\left(\bigotimes_{i=1}^{2g}  (\id_{D'_{\ait}}\ot \wt{S}_{\bitm})\circ \De_{\ait,\bitm}\right)$$
is filtration-preserving (the $t^{-g|\La_l|}$ factor is simply to eliminate the powers of $t$ from the $G'$-pivots in $\wt{S}$). By Proposition \ref{prop: SS formula}, it follows that $t^{-g|\La_l|}Z_{\uDHp}^{\rhoc}(K_o)\in D'_{\b_1^{-1}\ot\tt}[2gN]$ and by definition of our filtration, this implies that $\ZHrhoc(K_o)$ is a polynomial in $H^*\ot H$ of degree $\leq 2gN$. This proves our theorem. \hfill \qedsymbol

\begin{remark}
The construction of $\ZHrhoc(T)$ only requires $H$ to be $\Z^m$-graded (instead of $\N^m$-graded). Proposition \ref{prop: SS formula} is still valid in this case. Thus, we expect a genus bound to hold in this setting too. We prefer to work in the $\N^m$-graded case since the most relevant examples, Borel parts of quantum groups, are $\N^m$-graded. %Note that for applications to quantum group invariants, our theorem has to be applied to Borel parts, which are $\N^m$-graded.
\end{remark}

%EXTENSION FOR INFINITE DIM H: the only thing we use is the $H$-degree of $U\in D(H)^{\ot 2g}$ where $U$ is the invariant of the tangle which is the ``core" of the Seifert surface. Thus a more general thm could be $$\deg \ Z^{\rhoc}(K)\leq 2g(K)\cdot \deg_H(U)$$where $\deg_H(\otimes_{i=1}^{2g}p_i\ot h_i)=\sum_i|h_i|$. But what is $U$ going to be in this case? An element in a completion? does it really has a degree like that? Moreover: is this degree an invariant of the Seifert surface? of the knot?

%\subsection{Corollaries} Our theorem implies Friedl-Kim's genus bound for twisted Alexander polynomial \cite{FK:Thurston}. Indeed, as shown in \cite{LN:TDD}, when $H=\La(\C^n)$ is an exterior algebra, $\Aut(H)=GL(n,\C)$ and $$\e_{D(H)}(\ZHrhoc(T))=\tau^{\rhoc}(S^3\sm K,m)$$ where $m\sb \p(S^3\sm K)$ is a meridian and $\rho:\pi_1(X_K)\to SL(n,\C)$ is a homomorphism. Now, we have $$\tau^{\rhoc}(S^3\sm K,m)=\det(t\rho(m)-I_n)\tau^{\rhoc}(S^3\sm K)$$ so our theorem implies that $$n+\deg(\tau^{\rhoc}(S^3\sm K))\leq 2gn$$ which is equivalent to $$\deg(\tau^{\rhoc}(S^3\sm K))\leq n||\phi||_T$$where $||\phi||_T$ is the Thurston norm of the generator of $\phi\in H^1(S^3\sm K)$. This is exactly the genus bound from \cite{FK:Thurston}.

\def\Uqq{U_{\qq}(\sl_2)}

\section{Relation to ADO invariants}\label{sect: ADO}

\def\wD{D}
\def\sl{\mathfrak{sl}}
\def\ovU{\overline{D}}
\def\uq{\overline{U}_q(\mathfrak{sl}_2)}
\def\UUq{\overline{U}_q(\mathfrak{sl}_2)}
\def\Hh{\mathcal{H}}\def\RR{\mathcal{R}}
\def\k{k}
\def\K{k}
\def\ka{\k^{\a}}\def\kb{\k^{\b}}\def\kc{\k^{\c}}\def\Da{D_{\a}}
\def\Dp{D(p)}
\def\ko{\ov{k}}
\def\Hp{H_p}

Let $H$ be the algebra generated by $\k,E$ such that $\k E=\qq E\k$ (Hopf algebra structure defined below). If $\qq$ is a primitive $2p$-th root of unity (with $p\geq 2$), which we simply take as $\qq=e^{\pi i/p}$, we let $H_p$ be the quotient of $H$ by $\k^{4p}=1, E^p=0$. This is $\N$-graded with $|E|=1, |k|=0.$ During this whole section we denote $\qq^{\a}=e^{\pi i \a/p}$ and $[\a]=\frac{\qq^{\a}-\qq^{-\a}}{\qq-\qq^{-1}}$ for any $\a\in\C$.

\subsection{Quantum $\sl_2$} Consider the algebra $\wD(H)$ with generators $E,F,\k^{\pm 1},\kappa^{\pm 1}$ satisfying 
\begin{align*}
    \k E&=\qq E\k,  & \kappa E&=\qq^{-1}E\kappa, &  \k F&=\qq^{-1}F\k, & \kappa F&=\qq F\kappa, \\
    [E,F]&=\frac{\k^2-\kappa^2}{\qq-\qq^{-1}}, & \k\kappa&=\kappa \k, & \k \k^{-1}&=1=\kappa\kappa^{-1},
\end{align*}
and Hopf algebra structure determined by 
\begin{align*}
     \De(E)&=E\ot \k^2+1\ot E, & \De(F)&=\kappa^{2}\ot F+F\ot 1, & \De(\k^{\pm 1})&=\k^{\pm 1}\ot \k^{\pm 1}, & \De(\kappa^{\pm 1})&=\kappa^{\pm 1}\ot\kappa^{\pm1},\\
     \e(E)&=\e(F)=0, & \e(k)&=\e(\kappa)=1, & S(E)&=-E\k^{-2}, & S(F)&=-\kappa^{-2}F, \\
     S(\k)&=\k^{-1}, & S(\kappa)&=\kappa^{-1}.
\end{align*}

Since $k\kappa$ is a central group-like in $D(H)$, the quotient $D(H)/I$, where $I$ is the ideal generated by $k\kappa-1$, is a Hopf algebra. The usual quantum group $\Uqq$ (at a root of unity) is the subalgebra generated by $E,F,K^{\pm 1}:=k^{\pm 2}$ in this quotient.
\medskip

\subsection{The Drinfeld double of $H_p$} The Drinfeld double $D(H_p)$ is the quotient of $\wD(H)$ by the relations
\begin{align*}
     E^p&=F^p=0, & \k^{4p}&=\kappa^{4p}=1, 
\end{align*}
see \cite{FGST:modular}. For convenience, we will denote by $e,f,\ko,\overline{\kappa}$ the images of $E,F,k,\kappa$ in $D(H_p)$. Let $\Dp$ be the quotient of $D(H_p)$ by the ideal generated by $\ko\overline{\kappa}-1$. We will also denote by $e,f,\ko$ the image in $\Dp$ of the generators of $D(H_p)$.
\medskip

The Hopf algebra $\Dp$ is quasi-triangular with $R$-matrix given by $$R=\frac{1}{4p}\sum_{i,j=0}^{4p-1}\qq^{-ij/2}\ko^i\ot \ko^j\cdot\RR$$
where $\RR=\sum_{m=0}^{p-1}c_me^m\ot f^m$ and $c_m=\frac{(\qq-\qq^{-1})^m}{[m]!}\qq^{m(m-1)/2}$ for each $m\geq 0$. Note that this $R$-matrix is simply the projection in $D(p)$ of the usual $R$-matrix of the double $D(H_p)$. Moreover, $D(p)$ has a ribbon element determined by $\gg=vu$ where $\gg=k^{2p+2}$ and $u$ is the Drinfeld element (for an explicit expression of $v^{-1}$ see \cite{FGST:modular}). Note that if $\bb:=\k^{p+1}$ and $\bbb\in H_p^*$ is defined by $\bbb(k)=\qq^{\frac{p+1}{2}}, \bbb(E)=0$ then $(\bbb,\bb)$ satisfies the hypothesis of the Kauffman-Radford theorem (see \ref{subs: ribbon elements}). The pivot $\gg=k^{2p+2}$ is the projection in $D(p)$ of the pivot $\bbeta\ot\bb\in D(H_p)$.

%Moreover, $\Dp$ has a ribbon element given by $$v=\frac{1-i}{2\sqrt{p}}\sum_{m=0}^{p-1}\sum_{j=0}^{2p-1}\frac{(q-q^{-1})^m}{[m]!}q^{-\frac{m}{2}+mj+\frac{1}{2}(j+p+1)^2}f^me^m \ko^{2j}$$
\medskip

Note that $\Dp$ is the usual braided extension of the restricted quantum group $\uq$ at $q=\qq$ ($\uq$ is the Hopf subalgebra of $D(p)$ generated by $e,f,\ko^2$ and is not braided)

%There is another choice: $b=k^{1-p}, \b=\kappa^{p-1}$. This is what Murakami and Brown-Dimofte-Garoufalidis-Geer use in their papers.

\subsection{The twisted Drinfeld double of $H_p$} Consider the action of $(\C,+)$ on $H_p$ given by $\a\mapsto \phi_{\a}\in\Aut(H_p)$ defined by $\phi_{\a}(k)=k, \phi_{\a}(e)=\qq^{2\a}e$, that is, take $t=\qq^{2\a}$ in the degree twist action (here $\qq^{\a}:=e^{i\pi \a/p}$ for any $\a\in\C$). For each $\a$, let $\Da$ be the quotient of $D(H_p)_{\phi_\a}$ by the ideal generated by $k\kappa-1$, in particular, $D_0$ is the Hopf algebra $\Dp$ defined previously. Then $\Da$ has generators $e,f,\ko$ subject to the same relations as in $D_0$ except for the relation $[e,f]=\frac{\ko^2-\ko^{-2}}{\qq-\qq^{-1}}$ which now becomes
\begin{equation}
\label{eq: twisted bracket}
    [e,f]=\frac{\ko^2-\qq^{-2\a}\ko^{-2}}{\qq-\qq^{-1}}.
\end{equation}

The coproduct $\De_{\a,\b}:D_{\a+\b}\to D_{\a}\ot D_{\b}$ is given by
\begin{equation}
\label{eq: twisted coproduct}
    \De_{\a,\b}(e)=\qq^{-2\a}1\ot e+e\ot \ko^2, \hspace{1cm} \De_{\a,\b}(f)=\ko^{-2}\ot f+f\ot 1.
\end{equation}

These computations are carried out in \cite{Virelizier:Graded-QG} (in a slightly different convention, obtained by applying $\id_{H^*}\ot\a$ on each $D_{\a}$ to our formulas). The $R$-matrix is $$R_{\a,\b}=\frac{1}{4p}\sum_{i,j=0}^{4p-1}\qq^{-ij/2}\ko^i\ot \ko^j\cdot\sum_{m=0}^{p-1}c_m \qq^{2\a m}e^m\ot f^m$$ where $c_m$ is as before. 
This is simply the projection in $D_{\a}\ot D_{\b}$ of the $R$-matrix of $\underline{D(H_p)}$ defined in Subsection \ref{subs: TDDs}.
The (left) cointegral of $H_p$ is $\Lambda_l=\sum_{i=0}^{4p-1}k^ie^{p-1}$ and $\phi_{\a}(\La_l)=\qq^{2\a(p-1)}\La_l$ so that $r_{H_p}(\phi_{\a})=\qq^{2\a(p-1)}$. This has an obvious square root $\sqrt{r_{H_p}(\phi_{\a})}=\qq^{\a(p-1)}$, hence the Hopf group-coalgebra $\{D_{\a}\}_{\a\in\C}$ is $\C$-ribbon with ribbon element $v_{\a}$ which is the image in $D_{\a}$ of
\begin{align*}
\qq^{-\a(p-1)}(\id_{H^*}\ot\a)(v)
\end{align*}
where $v$ is the ribbon element of $D(H_p)$ and the $G$-pivot is  $$\gg_{\a}=\qq^{-\a(p-1)}\ko^{2p+2}.$$ %(Note: I need to check that the ribbon of \cite{FGST:modular} is the one coming from the above choice of $b,\b$). If $\b(k)=q^{\frac{p+1}{2}}$ then $\b=\kappa^{-(p+1)}$ so in the quotient $k\kappa=1$ the pivot $\bbeta\ot\bb$ maps to $k^{2p+2}$ which is the pivot of \cite{BBGe:logarithmic}. Actually, the other choice $b=k^{3p+1}$ and $\b$ gives the same pivot of the double).

\def\ovUH{\ovU^H}

\subsection{Relation to unrolled quantum $\sl_2$} To relate the above twisted Drinfeld double to the usual quantum group $\Uq$ we multiply each $e\in D_{\a}$ by $\qq^{\a}$ and $\ko\in D_{\a}$ by $\qq^{\frac{\a}{2}}$. Let $E=\qq^{\a}e, F=f, \K=\qq^{\frac{\a}{2}}\ko$ (note that there is one $E,\K$ on each $D_{\a}$ but we do not include $\a$ in the notation). In these new generators, relation (\ref{eq: twisted bracket}) becomes $$[E,F]=\frac{\K^2-\K^{-2}}{\qq-\qq^{-1}}$$ and $\K^{4p}=\qq^{2p\a}$ while the other relations among the generators remain unchanged. The coproduct (\ref{eq: twisted coproduct}) becomes $$\De_{\a,\b}(E)=E\ot \K^2+\qq^{-\a}1\ot E, \hspace{1cm} \De_{\a,\b}(F)=\qq^{\a}\K^{-2}\ot F+F\ot 1$$ 

\noindent and the $R$-matrix \begin{align*}R_{\a,\b}&=\frac{1}{4p}\sum_{i,j=0}^{4p-1}\qq^{-ij/2}\ko^i\ot \ko^j\sum_{m=0}^{p-1}c_m\qq^{2\a m}e^m\ot f^m\\
&=\frac{1}{4p}\sum_{i,j=0}^{4p-1}\qq^{-ij/2-\a i/2-\b j/2}\K^i\ot \K^{j}\sum_{m=0}^{p-1}c_m\qq^{\a m}E^m\ot F^m\\
&=\frac{\qq^{\a\b/2}}{4p}\sum_{i,j=0}^{4p-1}\qq^{-\frac{1}{2}(\a+j)(\b+i)}\K^i\ot \K^{j}\sum_{m=0}^{p-1}c_m\qq^{\a m}E^m\ot F^m\end{align*}

%This looks pretty close to the Hopf group-coalgebra defined by the central extension $(k^{4p})\sb D(H)/(k\kappa-1)$ (which is an extension of the Hopf group-coalgebra defined by the more usual central extension $(K^{2p})\sb \Uq$ whenever $q^{2p}=1$), except for the coproduct structure and the $R$-matrix. To relate these two Hopf group-coalgebras, we proceed as follows.

\medskip

 %Let $\ovU$ be the quotient Hopf algebra $\ovU=\wD/(E^p,F^p,k-\kappa^{-1})$. The {\em unrolled quantum group} $\ovUH$ is the algebra obtained by adding a generator $H$ to $\ovU$ satisfying \begin{align*} [H,E]&=2E, & [H,F]&=-2F, & Hk&=kH, \\ \De(H)&=H\ot 1+1\ot H, & S(H)&=-H, & \e(H)&=0.\end{align*}This is rather an extension of the usual unrolled quantum group (which is the subalgebra generated by $E,F,H,k^2$). We now relate the twisted Drinfeld double to $\ovUH$. 
%This is an extension of the semi-restricted quantum group $\UUq$, which is the subalgebra of $\ovU$ generated by $E,F,k^2$. We will relate the twisted Drinfeld double of $H$ to $\ovU$.\medskip

Now let $A_{\a}=\C[(\C,+)]\ltimes D_{\a}$ be the Hopf group-coalgebra defined in Subsection \ref{subs: ss-product Aa's}. Clearly $\v_{1/2}k^{-1}$ is central in $A_{\a}$, hence we mod out by $\v_{1/2}=k$ and we keep denoting the quotient by $A_{\a}$. For each $\b\in\C$, we denote by $\kb$ the element of $A_{\a}$ corresponding to $\v_{\b/2}$. Thus, $A_{\a}$ is obtained by adding generators $\kb,\b\in \C$ to $D_{\a}$ satisfying 
\begin{align*}
\kb E&=\qq^{\b}E\kb, & \kb F&=\qq^{-\b}F\kb, & \k^{\b+\c}&=\k^{\b}\k^{\c},\\
\De^A_{\a,\b}(\kc)&=\kc\ot \kc, & \k^0&=1_{\a}, & \k^1&=\k.
\end{align*}
Here $E,F,k$ denote the above (rescaled) generators of $D_{\a}$. Then $\uA=\{A_{\a}\}_{\a\in\C}$ is a quasi-triangular Hopf group-coalgebra with trivial crossing and has an $R$-matrix given by $$R^A_{\a,\b}=(1\ot \k^{2\a})R_{\a,\b}$$ where $R_{\a,\b}$ is the above $R$-matrix of $\{D_{\a}\}$.
%\begin{align*} R'_{\a,\b}&=(1\ot \k^{2\a})R_{\a,\b}
%&=\frac{q^{\a\b/2}}{4p}\sum_{i,j=0}^{4p-1}q^{-\frac{1}{2}(\a+j)(\b+i)}\K^i\ot \K^{j+2\a}\sum_{m=0}^{p-1}c_mq^{\a m}E^m\ot F^m\end{align*}
The ribbon element becomes $v^A_{\a}=v_{\a}\cdot \k^{2\a}$ and the pivot element is the same $\gg_{\a}$ as before, see Subsection \ref{subs: ss-product Aa's}.

\medskip

%To bring this closer to the usual $\Uq$ we multiply each $e\in A_{\a}$ by $q^{\a}$ and $\kb$ by $q^{\frac{\a\b}{2}}$. Let $E=q^{\a}e, F=f, \K^{\b}=q^{\frac{\a\b}{2}}\k^{\b}$ (note that there is one copy of $E,\K^{\b}$ on each $\Da$ but we do not include $\a$ in the notation). In these new generators, relation (\ref{eq: twisted bracket}) becomes $$[E,F]=\frac{\K^2-\K^{-2}}{q-q^{-1}}$$ and $\K^{4p}=q^{2p\a}$ while the other relations remain unchanged. The coproduct (\ref{eq: twisted coproduct}) becomes $$\De_{\a,\b}(E)=E\ot \K^2+q^{-\a}1\ot E, \hspace{1cm} \De(F)=q^{\a}\K^{-2}\ot F+F\ot 1$$ 

%\noindent and the $R$-matrix \begin{align*}R'_{\a,\b}=\frac{1}{4p}\sum_{i,j=0}^{4p-1}q^{-ij/2-\a i/2-\b j/2-\a\b}\K^i\ot \K^{j+2\a}\sum_{m=0}^{p-1}c_mq^{\a m}E^m\ot F^m\\=\frac{q^{-\a\b/2}}{4p}\sum_{i,j=0}^{4p-1}q^{-\frac{1}{2}(\a+j)(\b+i)}\K^i\ot \K^{j+2\a}\sum_{m=0}^{p-1}c_mq^{\a m}E^m\ot F^m\end{align*}

%The ribbon element becomes\begin{align*}v_{\a}=q^{-\a(p-1)}\frac{1-i}{2\sqrt{p}}\sum_{m=0}^{p-1}\sum_{j=0}^{2p-1}\frac{(q-q^{-1})^m}{[m]!}q^{-\frac{m}{2}+mj+\frac{1}{2}(j+p+1)^2-\a(j+\a)}F^mE^m\K^{2j+2\a}.\end{align*}%

\def\AJ{\uA^J}

As an algebra $A_{\a}$ is isomorphic to (a ribbon extension of) $\C[(\C,+)]\ltimes U_{\a}$, where $U_{\a}$ is the quotient of $\Uqq$ by $E^p=F^p=0, K^{2p}=\qq^{2p\a}$. However, the coproduct and the $R$-matrix are different. It turns out that the difference is measured by a Drinfeld twist. In our setting (grading by an abelian group, no group action) a {\em Drinfeld twist} in $\uA=\{A_{\a}\}_{\a\in\C}$ is a collection $J=\{J_{\a,\b}\}_{\a,\b\in\C}$ where each $J_{\a,\b}$ is an invertible element of $A_{\a}\ot A_{\b}$ satisfying
\begin{enumerate}
    \item $1_{\a}\ot J_{\b,\c}\cdot (\id_{A_{\a}}\ot \De^A_{\b,\c}(J_{\a,\b+\c}))=J_{\a,\b}\ot 1_{\c}\cdot (\De^A_{\a,\b}\ot\id_{A_{\c}}(J_{\a+\b,\c}))$,
    \item $\e_0\ot\id_{A_{\a}}(J_{0,\a})=1_{\a}=\id_{A_{\a}}\ot \e_0(J_{\a,0})$.
\end{enumerate}
for each $\a,\b,\c\in\C$. Drinfeld twisting a quasi-triangular Hopf group-coalgebra produces a new quasi-triangular Hopf group-coalgebra $\AJ=\{A_{\a}^{J}\}$ where $A_{\a}^J=A_{\a}$ as an algebra for each $\a$, the coproduct is $\De_{\a,\b}^J=J_{\a,\b}\De^A_{\a,\b}J_{\a,\b}^{-1}$ and the $R$-matrix is $$R_{\a,\b}^J=\tau_{\b,\a}(J_{\b,\a})\cdot R^A_{\a,\b}\cdot J_{\a,\b}^{-1}$$
where $\tau_{\b,\a}:A_{\b}\ot A_{\a}\to A_{\a}\ot A_{\b}$ is the usual symmetry (with signs in the super case). %{\color{blue} proof, pivot? I need the Drinfeld element after twisting, maybe only in specific case below.}

\begin{lemma}
Let $J_{\a,\b}=1 \ot \ka\in A_{\a}\ot A_{\b}$ for each $\a,\b\in\C$. Then $J=\{J_{\a,\b}\}_{\a,\b\in\C}$ is a (graded) Drinfeld twist on $\uA$.
\end{lemma}
\begin{proof}
It is clear that $J$ is a Drinfeld twist because the $\ka$'s are group-like. 
\end{proof}

Thus, we obtain a new ribbon Hopf group-coalgebra $\AJ$ whose coproduct is determined by
\begin{align*}
\De_{\a,\b}^J(E)=J_{\a,\b}\De^A_{\a,\b}(E)J_{\a,\b}^{-1}&=E\ot \K^2+\qq^{-\a}1\ot \ka E\k^{-\a}\\
&= E\ot \K^2+1\ot E
\end{align*}
and 
\begin{align*}
\De_{\a,\b}^J(F)=J_{\a,\b}\De^A_{\a,\b}(F)J_{\a,\b}^{-1}&=\qq^{\a}\K^{-2}\ot \ka F\k^{-\a}+F\ot 1\\
&=\K^{-2}\ot F+F\ot 1.
\end{align*}

%Thus $A^J$ is isomorphic, as a Hopf group-coalgebra, to $\{\ovU_{\a}\}_{\a\in\C}$ where $\ovU_{\a}=\ovU/(\K^{4p}=q^{2p\a})$, obtained by adding $\kb$'s to each $\ovU_{\a}$, that is $$A^J\cong \{U_{\a}\ltimes (\C,+)\}_{\a\in\C}.$$\medskipLet's say a module $V$ over $A_{\a}$ is a {\em weight module} if the $k^{\b}$'s are diagonalizable for each $\b\in\C$. But then $V=\oplus_{\chi}V_{\chi}$ where $\chi:\C\to\C^{\t}$ is a character and $k^{\b}v=\chi(\b)v$ for each $v\in V_{\chi}$. But $\chi(z)=e^{cz}$ for some $c\in \C$.is simply a weight module over $U_{\a}$. We now check that the $R$-matrices coincide. 

The antipode of $\AJ$ is $S^J_{\a}(x)=k^{\a}S^A_{\a}(x)k^{-\a}$. The $R$-matrix of $\AJ$ is
\begin{align*}
R_{\a,\b}^J&=(\kb\ot 1)R^A_{\a,\b}(1\ot \k^{-\a})\\
&=(\kb\ot 1)\frac{\qq^{\a\b/2}}{4p}\sum_{i,j=0}^{4p-1}\qq^{-\frac{1}{2}(\a+j)(\b+i)}\K^i\ot \K^{j+2\a}\sum_{m=0}^{p-1}c_m\qq^{\a m}E^m\ot F^m(1\ot \k^{-\a})\\
&=\frac{\qq^{\a\b/2}}{4p}\sum_{i,j=0}^{4p-1}\qq^{-\frac{1}{2}(\a+j)(\b+i)}\K^{i+\b}\ot \K^{j+\a}\sum_{m=0}^{p-1}c_mE^m\ot F^m\\
&=\mathcal{H}_{\a,\b}\cdot\mathcal{R}^J
\end{align*}
%But this exactly the $R$-matrix of the unrolled quantum group (except for the $q^{-\a\b/2}$ factor).
%\begin{proposition}The category $\Rep(A^J)$ is equivalent, as a braided category, to the category of weight modules over the (ribbon extension of the) unrolled quantum group $\Uq^H$ at a $2p$-th root of unity.\end{proposition}

%\begin{remark}We could simply say that $\Rep(A)$ is equivalent to $\Rep(\Uq^H)$, but we would need to include $J$'s in the monoidal structure of the equivalence. In the above statement, the isomorphisms $F(X)\ot F(Y)\to F(X\ot Y)$ (where $F$ is the equivalence) are all identities.\end{remark}

where $$\Hh_{\a,\b}=\frac{\qq^{\a\b/2}}{4p}\sum_{i,j=0}^{4p-1}\qq^{-\frac{1}{2}(\a+j)(\b+i)}\K^{i+\b}\ot \K^{j+\a}$$
and $\mathcal{R}^J=\sum_{m=0}^{p-1}c_mE^m\ot F^m.$

\begin{lemma}
The Drinfeld element, ribbon element and pivot of $\AJ$ are the same of $\uA$.
\end{lemma}
\begin{proof}
We only check this for the Drinfeld element. Write $R^A_{\a,\b}=\sum s'_{\a}\ot t'_{\b}$ so that $R^J_{\a,\b}=\sum k^{\b}s'_{\a}\ot t'_{\b}k^{-\a}$. Then
\begin{align*}
    u_{\a}^J&=\sum S^J_{-\a}(t'_{-\a}k^{-\a})k^{-\a}s'_{\a}\\
    &=\sum k^{\a}S^J_{-\a}(t'_{\a})k^{-\a}s'_{\a}\\
    &=\sum S_{-\a}(t'_{\a})s'_{\a}=u_{\a}
\end{align*}
where we used the above expression for $S^J$ in the third equality.
\end{proof}

The Hopf group-coalgebra $\{A_{\a}^J\}$ is isomorphic, as a Hopf group-coalgebra, to (a ribbon extension of) $\{\C[(\C,+)]\ltimes U_{\a}\}_{\a\in\C}$. The latter appears in \cite{Ohtsuki:colored} as an example of ``colored" Hopf algebra and is behind the definition of the non-semisimple quantum invariants of knots as we will see below. For the moment, we show the following.

%This is Ohtsuki's $R$-matrix from \cite{Ohtsuki:colored} except for the $q^{-\frac{\a\b}{2}}$ factor. But it is easy to see that if $\chi$ is a bicharacter on a group, then $\chi(\a,\b)R_{\a,\b}$ is also an $R$-matrix and the resulting invariants are equivalent to the original ones (because $G$ is abelian so we only use $R_{\a,\a}$'s and then one is a multiple of the other by $\chi(\a,\a)^m$). Actually, the ribbon element $v_{\a}$ gets multiplied by $\chi(\a,\a)$ as well so the unframed universal invariant $v^{-w(D)}Z$ gets unchanged (the $R$-matrices multiply $Z$ by $q^{\a^2w(D)}$ but the $v^{-w(D)}$ adds a $q^{-\a^2w(D)}$ factor). The present is the case with $\chi(\a,\b)=q^{-\frac{\a\b}{2}}$ (on $(\C,+)$).

\medskip

\begin{lemma}
\label{lemma: ADO proof - Drinfeld twist invariance}
The $\uA$-valued and the $\AJ$-valued universal invariants of a knot are equal.
\end{lemma}
\begin{proof}
It is a general fact that Drinfeld twisting does not changes the resulting quantum invariants. This is obvious if one understands a Drinfeld twist as an equivalence of monoidal categories. In the present case, we can see this directly as follows: let $\s^J_i:A_{\a}^{\ot n}\to A_{\a}^{\ot n}$ be the braid morphism corresponding to the $i$-th generator of the braid group $B_n$ coming from the $R$-matrix $R_{\a,\a}^J$. Let $\s_i$ be the one for $R^A_{\a,\a}$. Then if $J=1\ot k^{\a}\ot \k^{2\a}\ot\dots\ot k^{\a(n-1)}$ it is easy to see that $J^{-1}\s_i^JJ=\s_i$ for all $i$, hence $J^{-1}\s(\b)^JJ=\s(\b)$ for any braid $\b\in B_n$ where $\s(\b)^J,\s(\b)$ are the braid operators associated to $\AJ,\uA$ respectively. Closing the $i$-th strand in $J^{-1}\s(\b)^JJ$ has the effect of multiplying $k^{(i-1)\a}$ with $k^{-(i-1)\a}$, so closing the last $n-1$ strands is invariant under conjugation by $J$. Since $\AJ$ and $\uA$ have the same pivot, it follows that the invariants are equal.

\end{proof}

\subsection{Verma modules} For each $\a\in\C$, the Verma module $V_{\a}$ is the $\Uqq$-module with basis $v_0,\dots, v_{p-1}$ and action given by 
\begin{align*}
Ev_i&=[\a-i+1]v_{i-1}, & Fv_i&=[i+1]v_{i+1}, & K v_i&=\qq^{\a-2i}v_i,
\end{align*}
see \cite{Murakami:ADO}. This becomes a $D(H)/(k\kappa-1)$-module if we set $kv_i =\qq^{\frac{\a}{2}-i}v_i$. Since $\K^{4p}$ acts by $\qq^{2p\a}$ this is a module over $\Da$. It is a module over $A_{\a}=\C[(\C,+)]\ltimes D_{\a}$ if we set $\kb v_i=\qq^{\b(\a/2-i)}v_i$. The operator $\Hh_{\a,\b}\in A_{\a}\ot A_{\b}$ above acts over a tensor product $V_{\a}\ot V_{\b}$ as follows
\begin{align*}
\Hh_{\a,\b}(v_a\ot v_b)&=\frac{\qq^{\a\b/2}}{4p}\sum_{i,j=0}^{4p-1}\qq^{-\frac{1}{2}(\a+j)(\b+i)}\qq^{(i+\b)(\a/2-a)}\qq^{(j+\a)(\b/2-b)}(v_a\ot v_b)\\
&=\frac{\qq^{\a\b}}{4p}\sum_{i,j=0}^{4p-1}\qq^{-a\b-ai-b\a-bj-ij/2}(v_a\ot v_b)\\
&=\frac{\qq^{\a\b}}{4p}\sum_{i=0}^{4p-1}\qq^{-a\b-ai-b\a}\sum_{j=0}^{4p-1}(\qq^{-b-i/2})^j(v_a\ot v_b)\\
&=\frac{\qq^{\a\b}}{4p}\qq^{-a\b-a(-2b)-b\a}(4p)(v_a\ot v_b)\\
&=\qq^{\a\b-a\b+2ab-b\a}(v_a\ot v_b)\\
&=\qq^{\a\b/2}\qq^{H\ot H/2}(v_a\ot v_b)
\end{align*}

where $\qq^{H\ot H/2}$ is the operator defined by $$\qq^{H\ot H/2}(v_a\ot v_b)=\qq^{(\a-2a)(\b-2b)/2}v_a\ot v_b.$$

\def\ovR{\ov{R}}

Thus, the action of $R^J_{\a,\b}=\Hh_{\a,\b}\cdot\mathcal{R}$ over a tensor product $V_{\a}\ot V_{\b}$ is given by $$R^J_{\a,\b}=\qq^{\a\b/2}\ovR$$

where $\ovR:V_{\a}\ot V_{\b}\to V_{\a}\ot V_{\b}$ acts by $\qq^{\frac{H\ot H}{2}}\RR$.
%\begin{remark}{\color{blue} Unnecessary remark}Note that the power of $q$ ($e^{i\pi/p}$) in $\K v_i=q^{\a/2-i}v_i$ is, in principle, only defined up to $\pm 2np, n\in\Z$. But the fact that all $\k^{\b}$ act on $V_{\a}$ guarantee that $\a/2-i$ is well-defined as a complex number, which is needed to define $q^{H\ot H/2}$. Indeed, if $v_i$ is an eigenvector for each $k^{\b}$ then $k^{\b}v_i=q^{c\b}v_i$ for a unique $c\in\C$ and $c\b=\b(\a/2-i)$ (mod $2p$) for each $\b\in\C$ forces $c=\a/2-i$.  More generally, one can easily show that the category of weight modules over $A^J$ is (braided) equivalent to the category of weight modules over the unrolled quantum group, but we won't need this more general statement.\end{remark}
%Thus, our $R$-matrix acts on Verma modules exactly as the $R$-matrix defining the unrolled quantum group, except for the $q^{-\a\b/2}$ factor (which does not affects the invariants as mentioned above).

%{\color{blue} Note: Mention that $q^{\a\b/2}$ does nothing to the invariants.}

%{\color{blue} action of the ribbon, pivot over Verma modules.}

\subsection{ADO invariants}

\def\ovc{\ov{c}}

We define the ADO invariants following Murakami \cite{Murakami:ADO}. Let $K$ be a framed oriented knot which is the closure of a framed oriented $(1,1)$-tangle $K_o$. We suppose $K_o$ is given by a diagram $D$ with only upward crossings and left/right caps and cups and with the blackboard framing. Consider the above Verma module $V=V_{\a}$ with the Yang-Baxter operator $\ovc:V\ot V\to V\ot V$ induced from $\ovR$ above (that is $\ovc=P\circ \ovR$ where $P$ is the swap $v\ot w\mapsto w\ot v$), this is the same Yang-Baxter operator used in \cite{Murakami:ADO}. Then, to each positive crossing of $D$ we associate $\ovR$ or its inverse if the crossing is negative. On left caps and cups we simply associate the evaluation/coevaluation of vector spaces. On right caps and cups we use $$\lb v_i,v^*_i\rb=\qq^{\a(1-p)-2i}, \ 1\mapsto \sum_{i=0}^{p-1} \qq^{-\a(1-p)-2i}v^*_i\ot v_i.$$
\def\udp{\underline{D(H_p)}}

Pasting all these maps together as determined by the diagram $D$ defines a map $V_{\a}\to V_{\a}$. For generic $\a$, it can be shown that this map is multiplication by a scalar in $\qq^{\a^2w(D)/2}\cdot\Z[\qq,\qq^{\pm\a}]$ where $w(D)$ is the writhe of the diagram. This is called the ADO invariant of the framed knot $K$ and denoted $ADO'_p(K)$ (usually, the ADO invariant is a renormalized version of this that is no longer a polynomial, but since we only consider knots we prefer the unnormalized version given here). It is easily seen that a positive twist acts by multiplication by $\qq^{\a^2/2-\a(p-1)}$ on $V_{\a}$. It follows that
$$\qq^{-w(D)\a^2/2}\qq^{w(D)\a(p-1))}ADO'_p(K)$$
is an invariant of the underlying unframed knot $K$, where $w(D)$ is the writhe of the diagram. Moreover, this invariant belongs to $\Z[\qq,\qq^{\pm\a}]$ so setting $t=\qq^{2\a}$, this determines a polynomial in $\Z[\qq][t^{\pm 1/2}]$ which we denote simply by $ADO_p(K,t)$. Note that some authors use $t=\qq^{\a}$ instead. This polynomial is not symmetric in $t$, but it is after substituting $t=\qq^{-2}x$ \cite{MW:unified}. Recall from (\ref{eq: knot poly abelian case}) that $P_{H_p}^{\tt}$ is defined as $$P_{H_p}^{\tt}(K)=t^{w(D)(p-1)/2}\e_{D(\Hp)}(Z^{\theta}_{\udp}(K_o)).$$

\medskip

%The ADO invariant is defined by $$ADO_p(K)=q^{-\frac{1}{2}p(1-p)}\frac{q^{\a+1}-q^{-\a-1}}{q^{p\a}-q^{-p\a}}\lb T\rb_{\a}$$where $\lb T\rb_{\a}$ is the scalar from RT invariant of $K$ colored with $V_{\a}$. Since we will work with knots, we will only use $\lb T\rb_{\a}$ which we will denote $ADO'_p(K)$. This is a polynomial in $t=q^{\a}$.\medskipNote: More detais about ADO def. when is $V_{\a}$ simple over unrolled and all that?

\begin{proposition}
We have $ADO_p(K,t)=P_{H_p}^{\tt}(K)$ where $\theta$ is the degree twist of $H_p$.
\end{proposition}

% So ADO is usually a. polynomial in t^2, this is indeed as the computations of BDGG show.

\begin{proof}
In the above definition of ADO invariants, we used the Yang-Baxter operator $\ovc=\qq^{-\a^2/2}c^J_{\a,\a}$. Since the $\qq^{-\a^2/2}$ factor also appears on the ribbon element (equivalently, on the action given by a positive twist) it does not affects the invariant of an unframed knot, so we could equally use $c^J_{\a,\a}$ to define ADO. Using $c^J_{\a,\a}$, a positive twist acts by multiplication by $\qq^{\a^2-\a(p-1)}$. The maps associated to right caps and cups above are exactly given by multiplication by $\gg_{\a}^{\pm 1}$, which is the pivot of $\AJ$. Thus, the universal invariant $Z^J_{\a}\in A_{\a}$ defined from $\AJ$ determines the ADO invariant by $$ADO_p(K,\qq^{\a})v_0=\qq^{w(D)(-\a^2+\a(p-1))}Z_{\a}^Jv_0=\qq^{w(D)(-\a^2+\a(p-1))}Z_{\a}v_0$$

\noindent where the second equality is Lemma \ref{lemma: ADO proof - Drinfeld twist invariance}. Now, by (\ref{eq: ZA fom ZDH}), $Z_{\a}$ is of the form $$Z_{\a}=\pi(Z^{\theta}_{\udp}|_{t=\qq^{2\a}})\cdot k^{2\a w(D)}$$  where $\pi:D(H_p)_{\a}\to D_{\a}$ is the projection (recall that we mod out by $k\kappa-1$). Note that $Z_{\udp}^{\tt}$ is evaluated at $t=\qq^{2\a}$ since $A_{\a}$ is defined through the action $\phi_{\a}(e)=\qq^{2\a}e$, that is, $t=\qq^{2\a}$ on the degree twist action. Write $Z^{\theta}_{\udp}(K)=\sum \la_{ijl}k^i\kappa^jF^lE^l$ (rescaled generators of $D(H_p)_{\a}$), where $\la_{ijl}\in\C[t^{\pm 1}]$. Using $E^iv_0=0$ ($i>0$) we get
\begin{align*}
    Z_{\a}v_0=\sum_{i,j,l} \la_{ijl} k^{i-j}F^lE^lk^{2\a w(D)}v_0=\qq^{w(D)\a^2}\sum_{i,j} \la_{ij0}\qq^{\a(i-j)/2}v_0.
\end{align*}
But $k=\qq^{\a/2}\ov{k}$ and $\e(\ov{k})=1$, so $\e(k^{i-j})=\qq^{\a(i-j)/2}$. Since also $\e(E)=\e(F)=0$, it follows that $\e(Z^{\tt}_{\udp}(K))|_{t=\qq^{2\a}}v_0=Z_{\a}v_0$. Setting $\qq^{2\a}=t$ this implies the theorem.

%The (twisted) universal invariant $Z=Z^{\theta}_{D(H)}$ has the form $Z=\sum \la_j k^jf^ie^i$ where $k,e,f$ are the unscaled generators above and $\la_j\in \C[t^{\pm 1}]$. Since the $R$-matrices and ribbon element coincide (up to Drinfeld twist) as shown above, the ADO' invariant is obtained by applying this to $v_0\in V_{\a}$ (Note: here I use simplicity. True for all $\a$?). But $E^iv_0=0$ for $i\geq 1$ so we get $Zv_0=\sum \la_jk^jv_0$. Now, the unscaled $k$ satisfies $kv_i=q^{-i}v_i$ so $kv_0=v_0$, so $ADO=\sum \la_j$. But this is the same as $\e_{D(H)}(Z)$.
\end{proof}

%This means that ADO is a Fox-calculus-twisted Kuperberg invariant as in \cite{LN:twisted} though in a non-involutory setting. 

%Here is the symmetry thing: Martel-Willets show that $$ADO_p(K,x)=ADO_p(K,q^{-2}x^{-1})$$where $x=q^{\a}$. Thus, the symmetric ADO is $ADO(K,q^{-1}x)$. In our case we have a polynomial $P(t)$ in $t=q^{2\a}=x^2$. To get the symmetric ADO we set $t=q^{-2}x^2$.

%\subsection{Higher rank quantum groups} The proof above can be extended to show that the invariants $P^{\theta}_H(K)$ are ``ADO-like". The analogue of the Verma module is $D_{\theta}(H)\ot_{H}\kk_{\e}$ where $\kk_{\e}=\kk$ with $H$-module structure given by $\e$. However, to our knowledge, higher rank ADO have only been considered by Harper in the case of quantum $\mathfrak{sl}_3$ at a 4-th root of unity \cite{Harper:sl3-invariant}. { add https://arxiv.org/pdf/2207.01882.pdf}This case is already of interest (see Section EXAMPLES below) so briefly sketch a proof here. Moreover, the proof is slightly simpler since the order of the root of unity is coprime to the determinant of the Cartan matrix, so we do not need to extend our (restricted) quantum groups to make them braided.\medskip What about quantum supergroups? Do we recover Links-Gould invariant?

\bibliographystyle{amsplain}
\bibliography{/Users/daniel/Desktop/TEX/bib/referencesabr}

\end{document}